\documentclass[a4paper,11pt]{amsart}

\usepackage{mathtools}
\usepackage{amsmath}
\usepackage{amsfonts}
\usepackage{amssymb}
\usepackage{graphicx}
\usepackage{mathabx}
 \usepackage[colorlinks,citecolor=magenta,linkcolor=black]{hyperref}

\usepackage{color}
\usepackage{soul,xcolor}
\setstcolor{magenta}
\usepackage{microtype}
\usepackage{comment}

\usepackage{amsthm}
\usepackage[all,cmtip]{xy}
\usepackage{tikz-cd}
\usepackage{tikz}
\usetikzlibrary{matrix,arrows,decorations.pathmorphing}

\newcommand{\Fp}{\mathbb{\overline{F}}_p}

\newcommand{\QQ}{\mathbb Q}

\newcommand{\mbN}{\mathbb{N}}
\newcommand{\mbQ}{\mathbb{Q}}

\newcommand{\mbZ}{\mathbb{Z}}
\newcommand{\mbP}{\mathbb{P}}

\newcommand{\mcD}{\mathcal{D}}
\newcommand{\mcE}{\mathcal{E}}

\newcommand{\mcI}{\mathcal{I}}
\newcommand{\mcL}{\mathcal{L}}

\newcommand{\mcO}{\mathcal{O}}

\DeclareMathOperator{\Chow}{Chow}
\DeclareMathOperator{\Supp}{Supp}
\DeclareMathOperator{\Spec}{Spec}
\DeclareMathOperator{\Sing}{Sing}
\DeclareMathOperator{\codim}{codim}

\DeclareMathOperator{\Pic}{Pic}

\DeclareMathOperator{\Exc}{Exc}

\newcommand*{\defeq}{\mathrel{\mathop:}=}

\theoremstyle{plain}
\newtheorem{thmi}{Theorem}     
\newtheorem{cori}{Corollary}     
\newtheorem{theorem}{Theorem}[section]
\newtheorem{proposition}[theorem]{Proposition}
\newtheorem{lemma}[theorem]{Lemma}
\newtheorem*{lemma*}{Lemma}
\newtheorem{corollary}[theorem]{Corollary}

\theoremstyle{definition}
\newtheorem{definition}[theorem]{Definition}
\newtheorem{example}[theorem]{Example}

\theoremstyle{remark}
\newtheorem{remark}[theorem]{Remark}

\title[On the canonical bundle formula]
{On the canonical bundle formula and log abundance in positive characteristic} 
 
\author{Jakub Witaszek} 
\address{Department of Mathematics, Imperial College, London, 180 Queen's Gate, 
London SW7 2AZ, UK} 
\email{j.witaszek14@imperial.ac.uk}

\begin{document}
\maketitle
\begin{abstract}
We show that a weak version of the canonical bundle formula holds for fibrations of relative dimension one. We provide various applications thereof, for instance, using the recent result of Xu and Zhang, we prove the log non-vanishing conjecture for three-dimensional klt pairs over any algebraically closed field $k$ of characteristic $p>5$. We also show the log abundance conjecture for threefolds over $k$ when the nef dimension is not maximal, and the base point free theorem for threefolds over $\Fp$ when $p>2$.
\end{abstract}

\section{Introduction}
Since the beginning of algebraic geometry, mathematicians have looked for a way to classify all smooth projective varieties. A conjectural framework for such a classification called the Minimal Model Program (MMP for short) was established around forty years ago for the category of varieties with Kawamata log terminal (klt) singularities, and a big part of the program was shown to hold in characteristic zero at the beginning of this century (see \cite{bchm06}). However, in characteristic $p>0$ the most basic results of the Minimal Model Program are still widely open, notwithstanding the recent progress in the case of surfaces (see \cite{tanaka12}) and threefolds when $p>5$ (see \cite{hx13}, and also \cite{ctx13}, \cite{birkar13}, and \cite{bw14}).

One of very important tools used in the characteristic zero birational geometry is the \emph{canonical bundle formula} describing the behaviour of canonical divisors under log Calabi-Yau fibrations. Not only does this formula allow for a calculation of the canonical ring of a klt variety by means of the MMP, but also plays a vital role in proofs of many fundamental results, like for instance the log abundance conjecture for threefolds over $\mathbb{C}$. Unfortunately, in general the canonical bundle formula is false in positive characteristic, even in a case of a smooth fibration from a surface to a curve (cf.\ Example \ref{ex:cbf_false}).

The goal of this article is to embody an idea that the canonical bundle formula in characteristic $p>0$ can be partially recovered, if we replace the base of a log Calabi-Yau fibration by a purely inseparable cover, and, even more importantly, to provide various applications thereof. 
\begin{thmi}[see Theorem \ref{theorem:canonical_bundle_formula_main}] \label{theorem:canonical_bundle_formula_main_intro} Let $(X,\Delta)$ be an $n$-dimensional projective log canonical pair defined over an algebraically closed field $k$ of characteristic $p>0$, let $\phi \colon X \to Z$ be a contraction with the generic fibre of dimension one, and let $L$ be a $\mbQ$-Cartier $\mbQ$-divisor on $Z$ satisfying $K_X+\Delta \sim_{\mbQ} \phi^* L$. Assume that $p>3$, or that $\Delta$ is big and $p>2$. 

Then, there exists a purely-inseparable morphism $f \colon T \to Z$ such that
\[
f^*L \sim_{\mbQ} tf^*K_Z + (1-t)(K_{T}+\Delta_T),
\]
for some rational number $0 \leq t \leq 1$ and an effective $\mbQ$-divisor $\Delta_T$ on $T$.
\end{thmi}
Tanaka (see \cite{tanakapathologies}) constructed a Mori fibre space in characteristic two and three with a non-klt base; in his example the singularity admits a purely-inseparable cover by a smooth variety. However, note that our theorem does not provide an insight into what the singularities of $(T,\Delta_T)$ are. 

In view of this result and Tanaka's work, it is compelling to speculate about the existence of a ``weighted scheme'' living ``in-between'' $T$ and $Z$ that would appropriately capture the geometry of $X$ and allow for the inductive iteration of the MMP. Alas, we do not know of a suitable category of such objects.\\

Amongst its various applications, the canonical bundle formula stands out in a synergy with nef reduction maps, that is rational maps $\phi \colon X \dashrightarrow Z$ which contract exactly these curves $C$ passing through very general points of $X$ which satisfy $L \cdot C=0$, where $L$ is a given nef line bundle. The dimension of $Z$ is called the nef dimension of $L$ and denoted by $n(L)$. For the purpose of tackling the base point free theorem and the abundance conjecture in positive characteristic, we derive the following result from Theorem \ref{theorem:canonical_bundle_formula_main_intro}.

\begin{thmi}[see Theorem \ref{theorem:canonical_bundle_formula_3fold}]\label{theorem:canonical_bundle_formula_3fold_intro} With assumptions as in Theorem \ref{theorem:canonical_bundle_formula_main_intro}, suppose that $K_X+\Delta$ is nef and $n(K_X+\Delta)=n-1$. Assume that the resolution of singularities and the Minimal Model Program are valid\footnote{see Remark \ref{rem:cbf_3fold} for a clarification} for $(n-1)$-dimensional varieties over $k$. Then, there exists a purely-inseparable morphism $f \colon T \to Z$ such that
\[
f^*L \sim_{\mbQ} K_T + \Gamma,
\]
for some pseudo-effective $\mbQ$-divisor $\Gamma$ on $T$.
\end{thmi}
\noindent One can show that if $(X,\Delta)$ is klt and $\Delta$ is big, then $\Gamma$ is big as well. In any case, the $\mbQ$-divisor $\Gamma$ is not canonically defined, and so there is no hope to control the singularities of $(T,\Gamma)$. \\

The main application of our canonical bundle formula pertains to the abundance and the non-vanishing conjectures. After the first version of this paper had been announced, Xu and Zhang showed the non-vanishing conjecture for terminal threefolds (\cite{xuzhang18}). Combining a generalisation of our initial results (regarding the case of adjoint divisors of  non-maximal nef dimension) with their breakthrough theorem, we obtain the non-vanishing conjecture for all three-dimensional klt pairs in characteristic $p>5$. To the best of our knowledge, the result is new even when the boundary is empty.
\begin{thmi}[Log non-vanishing] \label{thmi:non_vanishing} Let $(X,\Delta)$ be a projective Kawamata log terminal pair of dimension three defined over an algebraically closed field $k$ of characteristic $p>5$. Assume that $K_X+\Delta$ is pseudo-effective. Then $\kappa(K_X+\Delta) \geq 0$.
\end{thmi}

The abundance conjecture is one of the most fundamental open problems left in the study of the birational geometry of threefolds in characteristic $p>5$, and it has gathered much attention recently, leading to very interesting results by Das, Hacon, Waldron, and Zhang (see \cite{daswaldron,waldronabundance,zhang17}). It predicts that a log minimal model is either of general type or admits a ``log Calabi-Yau'' fibration over a lower dimensional variety. We can show the log abundance conjecture when the nef dimension is not maximal. 
\begin{thmi}[Log abundance for non-maximal nef dimension] \label{thmi:abundance}  Let $(X,\Delta)$ be a projective Kawamata log terminal pair of dimension three defined over an algebraically closed field $k$ of characteristic $p>5$. Assume that $K_X+\Delta$ is nef and $n(K_X+\Delta)\leq 2$. Then $K_X+\Delta$ is semiample.
\end{thmi}  
\noindent This, combined with \cite{zhang17}, shows log abundance when $\dim \mathrm{Alb}(X)>0$ (see Corollary \ref{cori:albabundance}). Theorem \ref{thmi:abundance} seems to be a valuable step towards the proof of the log abundance conjecture for klt three-dimensional pairs (see Remark \ref{remark:logabundance}).

The initial motivation for our study of the canonical bundle formula in positive characteristic stemmed from our interest in low characteristic birational geometry. Despite the substantial progress on the MMP in characteristic $p>5$, very little is known for $p\leq 5$. The geometry of threefolds in low characteristic seems surprisingly mysterious, and it is widely open whether the basic results such as the base-point-free theorem or the existence of flips should hold or not. 

Our main contribution in this direction is the proof of the base point free theorem over $\Fp$ when $p>2$; we build on the work of Keel who showed this result for big line bundles (but an arbitrary $p>0$) in the seminal paper \cite{keel99}. In particular, we get the existence of log Mori fibre spaces for threefolds over $\Fp$ when $p>2$.

\begin{thmi}[Base point free theorem over $\Fp$] \label{thmi:bpf} Let $(X,\Delta)$ be a three-dimensional projective log canonical pair defined over $\Fp$ for $p>2$ and let $L$ be a nef Cartier divisor such that $L - (K_X+\Delta)$ is nef and big. Then $L$ is semiample.
\end{thmi}
\noindent 
The above result generalises \cite[Theorem 1.2]{NW17}. Note that the base point free theorem for threefolds has been shown to hold over any algebraically closed field of characteristic $p>5$ by Birkar and Waldron (see \cite{bw14}).

We can also show the equality of the Kodaira dimension and the nef dimension in the base point free theorem in characteristic $p>2$ except when $L \equiv 0$. 

\begin{thmi}[Strong non-vanishing in the base point free theorem] \label{thmi:nonvanishing_bpf} Let $(X,\Delta)$ be a three-dimensional projective Kawamata log terminal pair defined over an algebraically closed field of characteristic $p>2$ and let $L$ be a nef Cartier divisor such that $L - (K_X+\Delta)$ is nef and big. When $n(L)=0$, assume $\dim \mathrm{alb}(X) \neq 1$. Then $\kappa(L) = n(L)$.
\end{thmi}
Although Theorem \ref{thmi:bpf} and Theorem \ref{thmi:nonvanishing_bpf} were known to hold before when $p>5$, we believe that they are interesting for two-fold reason: to the best of our knowledge, these are the first substantial positive results in the MMP for threefolds in low characteristic since the work of Keel (see \cite{cascinitanakaplt} and \cite{tanakapathologies} for the negative results), and further the proofs are completely independent of the existence of flips in positive characteristic.  \\

We finish off this part of the introduction by depicting the historical background. The canonical bundle formula in characteristic zero has been developed by Kodaira, Kawamata, Fujino, Mori, Shokurov, and Ambro, among others (cf.\ \cite{ambro04}). The case of contractions of relative dimension one with coefficients of $\Delta$ greater than $\frac{2}{p}$ in characteristic $p>0$ has been verified in \cite{ctx13} (see also \cite{hacondas}). The canonical bundle formula for when the generic fibre is F-split has been obtained in \cite{schwededas}, while an analogue of the canonical bundle formula for when the generic fibre is F-pure may be found in \cite[Theorem 1.4]{ejiri16a} (cf.\ \cite{patakfalvi12}).

The base-point-free theorem for threefolds over algebraically closed fields of characteristic $p>5$ has been proven in \cite{birkar13} and \cite{bw14} based on \cite{keel99}, \cite{hx13} and \cite{ctx13}. To the best of our knowledge, the idea to apply the nef reduction map in the context of the base point free theorem in positive characteristic was employed for the first time in the article of Cascini, Tanaka, and Xu (\cite{ctx13}), to which our work owns a great deal. In fact, a part of our article may be seen as a natural generalisation of \cite{ctx13} to the case of arbitrary coefficients and $p>2$ (cf.\ \cite[Theorem 1.9]{ctx13}). Some log canonical variants of the base-point-free theorem have been obtained in \cite{MNW15} (for $L$ big over $\Fp$), \cite{NW17} (for $p>5$ over $\Fp$), and \cite{waldronlc} (lc-MMP). Certain important special cases of the abundance conjecture have been proven in \cite{waldronabundance} (Kodaira dimension two), \cite{daswaldron} (Kodaira dimension one and partially for nef dimension zero), and \cite{zhang17} (when the Albanese dimension is non-zero and the boundary is empty and partially for klt boundaries). 

\subsection*{An idea of the proof of Theorem \ref{theorem:canonical_bundle_formula_main_intro} and Theorem \ref{theorem:canonical_bundle_formula_3fold_intro}}
We start by providing a sketch of the proof of Theorem \ref{theorem:canonical_bundle_formula_main_intro}. By the assumptions on the characteristic, we can assume that the generic fibre is smooth (see Proposition \ref{proposition:general_fibre_smooth}). Let $F$ be a general fibre of $\phi$. If $k$ was of characteristic zero, then $(F, \Delta|_F)$ would be log canonical. Unfortunately, this need not be true in positive characteristic, but in those cases when $(F,\Delta|_F)$ is indeed lc, then the theorem follows by the standard proof of the canonical bundle formula for fibrations of relative dimension one (see Proposition \ref{proposition:canonical_bundle_formula}).

Hence, we can assume that $(F,\Delta|_F)$ is not lc. As a consequence, we can show that $\Delta = aT + B$ for an irreducible divisor $T$ such that $\phi|_T$ is generically purely inseparable of degree $p^k$ for some $k\in \mbN$, an effective $\mbQ$-divisor $B$ such that $T \not \subseteq \Supp B$, and a rational number $a$ such that $0 < \frac{1}{p^k} < a \leq 1$. For simplicity, assume that $X$, $Z$, $T$ and $\phi$ are smooth, and $\phi|_T$ is purely inseparable. We are going to show that Theorem \ref{theorem:canonical_bundle_formula_main_intro} holds for $f = \phi|_T$. 

To this end, we choose $t \in \mbQ_{\geq 0}$ such that $a = \frac{1}{p^k}t + (1-t)$ and write 
\[
K_X + \Delta = t(K_X + \frac{1}{p^k}T) + (1-t)(K_X + T) + B.
\]
By taking a base change $X_T$ of $X$ via $f$ and applying the adjunction, one can show that $(K_X + \frac{1}{p^k}T)|_T \sim_{\mbQ} f^*K_Z$ (see Lemma \ref{lemma:inseparable_adjunction}). Therefore,
\[
f^*L \sim_{\mbQ} (K_X + \Delta)|_T \sim_{\mbQ} tf^*K_Z + (1-t)K_T + B|_T,
\]
which concludes the proof.\\

As for the proof of Theorem \ref{theorem:canonical_bundle_formula_3fold_intro}, we use the MMP to show that $P \defeq K_Z + \lambda L$ is pseudoeffective for some $\lambda > 0$ (see Lemma \ref{lemma:pseudoeffective_surfaces}). Then
\begin{align*}
(1+t\lambda)f^*L &\sim_{\mbQ} tf^*(K_Z + \lambda L) + (1-t)K_T + B|_T \\
&= (1-t)K_T + tf^*P + B|_T.
\end{align*}
Now, the theorem follows by composing $f$ with some power of Frobenius.\\

\subsection*{An idea of the proof of Theorems 3--6} In what follows, we assume for simplicity that the nef reduction map $\psi \colon X \dashrightarrow Z$ of $K_X+\Delta$ is proper.

Let us start with Theorem \ref{thmi:non_vanishing}. We can assume that $K_X+\Delta$ is nef by replacing $(X,\Delta)$ by its minimal model. Furthermore, by bounding the lengths of extremal rays (as in \cite{kkm94}), one can see that this theorem follows from the validity of the non-vanishing conjecture for terminal threefolds (shown in \cite{xuzhang18}) and the log non-vanishing conjecture for when $n(K_X+\Delta) \leq 2$. Thus, it is enough to show Theorem \ref{thmi:abundance}.

We split the proof of Theorem \ref{thmi:abundance} into three cases depending on the nef dimension. Assume $n(K_X+\Delta)=0$. If $\mathrm{Alb}(X)\neq 2$, then the log abundance  follows from \cite{zhang17}; if $\dim \mathrm{Alb}(X)=2$, we apply Theorem \ref{theorem:canonical_bundle_formula_main_intro} to the Albanese map (see Proposition \ref{prop:non_vanishing0} for details). When $n(K_X+\Delta)=1$, the proof is a simple consequence of the abundance for surfaces (see Proposition \ref{prop:non_vanishing1}).

We are left to show log abundance when $n(K_X+\Delta)=2$. A key component used to tackle this case is the following result, the proof of which is based on a careful study of the classification of surfaces.
\begin{lemma*}[{Lemma \ref{lem:bigsurface}, cf.\ \cite[Theorem 1.4]{MNW15}}] Let $Z$ be a normal $\mbQ$-factorial projective surface defined over an algebraically closed field $k$ and let $\Delta_Z$ be a pseudo-effective $\mbQ$-divisor such that $K_Z+\Delta_Z$ is nef of maximal nef dimension. Then $K_Z+\Delta_Z$ is $\mbQ$-effective up to numerical equivalence.
\end{lemma*}
Using Theorem \ref{theorem:canonical_bundle_formula_3fold_intro}, we can descend $K_X+\Delta$ via the nef reduction map $\psi \colon X \to Z$ to an adjoint divisor $K_Z + \Delta_Z$, with $\Delta_Z$ pseudo-effective, up to replacing $Z$ by a purely inseparable cover. Then the above lemma gives that $K_Z+\Delta_Z$, and so $K_X+\Delta$, are $\mbQ$-effective up to numerical equivalence.  

The proof of the semiampleness of $K_X+\Delta$ employs a special dlt modification, abundance for slc surfaces, and the evaporation technique. We refer the reader to Subsubsection \ref{ss:neftwo} for an explanation of the idea of the proof.\\

The proofs of Theorem \ref{thmi:bpf} and Theorem \ref{thmi:nonvanishing_bpf} are of a similar nature. The main difference is that when $n(L) = 3$, we apply \cite[Theorem 1.4]{ctx13} and (in the proof of the former theorem) the base-point-free theorem for big line bundles over $\Fp$ (see \cite[Theorem 1.1]{MNW15} and \cite[Theorem 0.5]{keel99}). 

\subsection{Outline of the paper} In Section \ref{s:preliminaries}, we gather preliminary results: we describe the behaviour of the relative canonical divisor under base change (Proposition \ref{proposition:relative_canonical_divisor_under_base_change}), recall special properties of varieties over $\Fp$ (Subsection \ref{ss:fp}), and provide a recap of basic facts about Chow varieties (Subsection \ref{ss:chow}) which we then use to define nef reduction maps over countable fields (Subsection \ref{subsection:nef_reduction_map}). In Section \ref{s:cbf}, we show Theorem \ref{theorem:canonical_bundle_formula_main_intro} and \ref{theorem:canonical_bundle_formula_3fold_intro}. In Section \ref{s:bpf}, we prove Theorem \ref{thmi:non_vanishing}, \ref{thmi:abundance}, \ref{thmi:bpf}, and \ref{thmi:nonvanishing_bpf}.

\subsection{Notation and conventions}
\begin{itemize}
\item We say that $X$ is a \emph{variety} if it is integral, separated, and of finite type over a field $k$. 
\item We call $(X,\Delta)$ a \emph{log pair} if $X$ is a normal variety and $\Delta$ is an effective $\mbQ$-divisor on $X$ such that $K_X + \Delta$ is $\mbQ$-Cartier. If $\Delta$ is not effective, then we call $(X,\Delta)$ a \emph{sub-log pair}.
\item We say that a morphism $f \colon X \to Y$ is a \emph{contraction} if $X$ and $Y$ are normal varieties, $f_*\mcO_X = \mcO_Y$, and $f$ is projective and surjective.
\item  We call a divisor $D \subseteq X$ \emph{horizontal} with respect to a contraction $f$ if $f|_D$ is dominant; if this is not the case, then we call $D$ \emph{vertical}. We denote the horizontal and the vertical part of a $\mbQ$-divisor $\Delta$ by $\Delta^h$ and $\Delta^v$, respectively. 
\item We say that a dominant morphism $f \colon X \to Y$ is \emph{purely-inseparable} if the extension of fractions fields $K(Y)/K(X)$ is purely-inseparable, in particular the identity is a purely-inseparable morphism.
\item We say that a dominant morphism $f \colon X \to Y$ is \emph{of relative dimension $k$} if the generic fibre $X_{\mu}$ is of dimension $k$, i.e.\ $\dim X - \dim Y = k$.
\item Let $D \subseteq X$ be a divisor on a scheme $X$, let $\pi \colon \overline{D} \to D$ be the normalisation of $D$, and let $L$ be any $\mbQ$-Cartier divisor on $X$. Then by abuse of notation we denote $\pi^*L|_D$  by $L|_{\overline{D}}$.
\item Given a reduced scheme $X$ of finite type over a field and its normalisation $\overline{X} \to X$, we define the \emph{conductor} $\mcD \subseteq \overline{X}$ as the scheme defined by the maximal possible ideal sheaf $\mcI$ satisfying $\mcI \mcO_{\overline{X}} \subseteq \mcO_X$.  
\item Let $D$ be a $\mbQ$-divisor on a scheme $X$. Then there exists a unique decomposition $D = D_{>0} + D_{<0}$ such that $D_{>0}$ is an effective $\mbQ$-divisor (which we call the \emph{positive part} of $D$) and $D_{<0}$ is an anti-effective $\mbQ$-divisor (which we call the \emph{negative part} of $D$). 
\item We say that a $\mbQ$-divisor $D$ on a normal variety $X$ is \emph{pseudo-effective} if for every ample $\mbQ$-Cartier $\mbQ$-divisor $A$ the sum $D+A$ is $\mbQ$-linearly equivalent to an effective $\mbQ$-divisor. Note that we do not assume that $D$ is $\mbQ$-Cartier. The image of a pseudo-effective divisor under any proper birational morphism is pseudo-effective.
\end{itemize}

\section{Preliminaries} \label{s:preliminaries}

We refer to \cite{km98} and \cite{kollar13} for basic definitions in birational geometry, and to \cite[Tag 0BRV]{stacks-project} for the theory of curves over arbitrary fields.

The observation that semiampleness can be verified after taking a purely inseparable cover plays a fundamental role in this paper.
\begin{lemma}[{cf.\ \cite[Lemma 2.10]{keel99}}] \label{lem:semiample_cover} Let $f \colon X \to Y$ be a proper surjective morphism of normal varieties. Then a line bundle $L$ is semiample if and only if $f^*L$ is semiample.
\end{lemma}

The following result says that Fano or Calabi-Yau fibrations of relative dimension one have smooth generic fibres in a sufficiently large characteristic.
\begin{proposition} \label{proposition:general_fibre_smooth} Let $X$ be a normal variety defined over a field of characteristic $p>0$, and let $f \colon X \to Z$ be a contraction of relative dimension one with the generic fibre $X_{\mu}$. Assume that 
\begin{enumerate}
	\item $\deg K_{X_{\mu}} < 0$ and $p>2$, or
	\item $\deg K_{X_{\mu}} = 0$ and $p>3$. 
\end{enumerate}
Then $X_{\mu}$ is smooth.
\end{proposition}
\begin{proof}
Since $X$ is normal, $X_{\mu}$ is a regular curve over $\Spec K(Z)$, where $K(Z)$ is the fraction field of $Z$. Given that $f_*\mcO_X = \mcO_Z$, we have $H^0(X_{\mu}, \mcO_{X_{\mu}}) = K(Z)$.  By \cite[Tag 0BY6, Lemma 49.8.3]{stacks-project}, the genus of $C$ is zero when $\deg K_{X_{\mu}} < 0$, and is one when $\deg K_{X_{\mu}} = 0$. Thus, the proposition follows from \cite[Lemma 6.5]{ctx13} and \cite[Proposition 2.9]{zhang17}. 
\end{proof}

We recall Raynaud's flatification.
\begin{theorem}[{\cite[Thm 5.2.2]{rg71}}] \label{thm:equidimensional}
Let $\pi \colon X \to S$ be a morphism of varieties such that there exists an open subset $U \subseteq S$ for which $\pi|_{\pi^{-1}(U)}$ is flat. 
Then, there exists a birational morphism $h \colon S' \to S$ and a commutative diagram
\begin{center}
\begin{tikzcd}
X' \arrow{r} \arrow{d}{\pi'} & X \arrow{d}{\pi} \\
S' \arrow{r}{h} & S,
\end{tikzcd}
\end{center}
such that $\pi'$ is flat, where  $X'$ is the closure of the generic fibre $X_{\mu}$ of $\pi$ in $X \times_S S'.$
\end{theorem}
 
In particular, if $S$ is normal, then there exists an open subset $U \subseteq S$, whose complement is of codimension at least two, such that $\pi$ is flat over $U$. Let us note that every morphism with equidimensional fibres between regular schemes is flat.

Last, we describe the behaviour of the relative canonical divisor under base change. 

\begin{proposition}[{cf.\ \cite[Proposition 2.1 and Theorem 2.4]{cz15}}] \label{proposition:relative_canonical_divisor_under_base_change} Let $\phi \colon X \to Z$ be a flat contraction between Gorenstein quasi-projective varieties. Let $f \colon T \to Z$ be a generically finite surjective morphism from a Gorenstein quasi-projective variety $T$ such that $X' \defeq X \times_Z T$ is reduced and let $\pi \colon Y \to X'$ be the normalisation of $X'$ sitting inside the following commutative diagram
\begin{center}
\begin{tikzcd}
Y \arrow{r}{\pi} \arrow{rd} \arrow[bend left=35]{rr}{h} & X' \arrow{r}{g} \arrow{d}{\phi'} & X \arrow{d}{\phi} \\
& T \arrow{r}{f} & Z.
\end{tikzcd}
\end{center}
Then $h^* K_{X/Z}= K_{Y/T} + D$, where $D$ is an effective divisor supported on the conductor of $\pi$. 
\end{proposition}

\begin{proof}
By \cite[Tag 0C02, Lemma 46.26.7]{stacks-project}, the morphism $\phi$ is Gorenstein. Thus $\phi'$ is Gorestein as well (\cite[Tag 0C02, Lemma 46.26.8]{stacks-project}), and so is $X'$ by \cite[Tag 0C02, Lemma 46.26.6]{stacks-project}. Further, \cite[Proposition 2.3]{cz15} implies that $g^*\omega_{X/Z} = \omega_{X'/T}$. We can conclude the proof by \cite[Proposition 2.3]{reid94} (cf.\ \cite[Proposition 2.10]{tanakainseparable}).
\end{proof}

\begin{remark} \label{rem:reduced} Note that $X'$ in the statement of Proposition \ref{proposition:relative_canonical_divisor_under_base_change} is reduced if and only if $X \times_{K(Z)} K(T)$ is reduced, where $K(Z)$ and $K(T)$ are the fraction fields of $Z$ and $T$, respectively. Indeed, since $\phi$ is a flat morphism between reduced schemes, it is $S_1$, and hence $\phi'$ is $S_1$ as well. Thus $X'$ has no embedded points.
\end{remark}

 Recall that a finite universal homeomorphism is a morphism which is a homeomorphism under any base change. In particular, purely inseparable morphisms are finite universal homeomorphisms.
\begin{lemma} \label{lem:Keel_universal_mor} Let $f \colon X \to Y$ be a finite universal homeomorphism between schemes of finite type over a field $k$ of characteristic $p>0$ and let $L$ be a nef Cartier divisor on $Y$. Then $L$ is semiample if and only if $f^*L$ is semiample. In particular, $L$ is semiample if and only if $L|_{Y^{\mathrm{red}}}$ is semiample where $Y^{\mathrm{red}}$ is the reduction of $Y$.
\end{lemma}
\begin{proof} 
This is \cite[Lemma 1.4]{keel99}. Note that the inclusion $Y^{\mathrm{red}} \subseteq Y$ is a finite universal homeomorphism.
\end{proof}

\subsection{Special properties of varieties over $\Fp$} \label{ss:fp}
We list some results on line bundles on varieties over $\Fp$ which are only valid over this field. 
\begin{lemma}[{\cite[Lemma 2.16]{keel99}}] \label{lemma:fp} Any numerically trivial line bundle on a variety $X$ defined over $\Fp$ is torsion.
\end{lemma}

On normal projective surfaces over $\Fp$, nef line bundles are semiample under very weak conditions. Note that such surfaces are always $\mbQ$-factorial (see \cite[Theorem 4.5]{tanaka12})
\begin{theorem}[{\cite[Theorem 2.9]{artin62}, cf.\ \cite[Corollary 0.3]{keel99}}] \label{thm:artin} If $D$ is a nef and big divisor on a normal projective surface $S$ defined over $\Fp$, then it is semiample.
\end{theorem}

\begin{theorem} \label{thm:mnw_surface} Let $S$ be a normal projective surface defined over $\Fp$ and let $\Delta$ be a pseudo-effective $\mbQ$-divisor on $S$. Assume that $K_S + \Delta$ is nef. Then it is semiample.
\end{theorem}
\begin{proof}
This is \cite[Theorem 1.4]{MNW15} as we can write $\Delta = N + P$ for a nef $\mbQ$-divisor $N$ and an effective $\mbQ$-divisor $P$ by Zariski decomposition (see \cite{lazarsfeld04a}).
\end{proof}

\subsection{Chow varieties} \label{ss:chow}
In this subsection, we recall basic properties of Chow varieties and algebraic equivalence relations closely following \cite[Chapter I]{kollar96}. The theory is used in Subsection \ref{subsection:nef_reduction_map} to construct nef reduction maps over countable algebraically closed fields. Let $X$ be a fixed projective variety defined over an algebraically closed field $k$. 

We define an \emph{effective $d$-dimensional algebraic cycle} on $X$ to be a finite formal sum $\sum a_i [V_i]$, with $a_i \in \mbN$, and $V_i$ irreducible, reduced, and closed $d$-dimensional subschemes.

\begin{definition}
A \emph{well defined family of effective proper $d$-dimensional algebraic cycles} of $X$ over a $k$-variety $W$, denoted $(g \colon U \to W)$, is an effective algebraic cycle $U = \sum m_i [U_i]$ on $X$ (here $U_i$ are its irreducible components) together with a proper morphism $g \colon \Supp U \to W$ to a reduced scheme $W$ such that for every $i$ the restriction $g|_{U_i}$ is surjective onto an irreducible component of $W$ with fibres of dimension $d$. Further, we assume a technical condition necessary for the existence of the cycle theoretic fibre $g^{[-1]}(w)$ at all $w \in W$ (see \cite[I, 3.10.4]{kollar96} for details).
\end{definition}


We say that a well defined family of algebraic cycles $(g \colon U \to W$) satisfies the \emph{field of definition condition} if $g^{[-1]}(w)$ is defined over $k(w)$ for every $w \in W$. In particular, if $W = \Spec k$ for an algebraically closed field $k$, then the field of definition condition is automatically satisfied.
Accordingly, we can define:
\begin{align*}
  {Chow}^{\mathrm{small}}(X)(Z) &\defeq \left\{ \begin{array}{lll}
               \text{Well defined algebraic families}\\
               \text{of proper effective cycles}\\
               \text{of } X \times Z \text{ over } Z \text{ which satisfy}\\
	    \text{the field of definition condition.}
            \end{array} \right\}.
\end{align*}
This is a well defined functor on the category of semi-normal schemes (\cite[Definition I, 7.2.1]{kollar96}).

Fix an ample Cartier divisor $H$ on $X$.  For $d,d' \in \mbN$, we define ${Chow}^{\mathrm{small}}_{d,d'}(X)$ to be the subfunctor of ${Chow}^{\mathrm{small}}(X)$ describing families of effective cycles of dimension $d$ and degree $d'$ with respect to $H$.
\begin{theorem}[{\cite[IV, Theorem 4.13]{kollar96}}] 
Let $X$ be a projective variety over a field $k$. Then there exists a semi-normal projective scheme $\mathrm{Chow}_{d,d'}(X)$ coarsely representing  ${Chow}^{\mathrm{small}}_{d,d'}(X)$.
\end{theorem}
Unfortunately, the universal family over $\mathrm{Chow}_{d,d'}(X)$ need not satisfy the field of definition condition, and so it is not a cycle in ${Chow}^{\mathrm{small}}_{d,d'}(X)$.\\

Now we turn our attention to equivalence relations defined by subschemes of $\mathrm{Chow}_{1,d'}(X)$ for a normal projective variety $X$. We refer to \cite[IV, Definition 4.2]{kollar96} for the definition of a proalgebraic relation. Given a proper subscheme $V \subseteq \mathrm{Chow}_{1,d'}(X)$ parametrizing connected $1$-cycles, one can construct a proalgebraic proper connected relation $R$ (see \cite[IV, Theorem 4.8]{kollar96}) which identifies two points $x_1, x_2 \in X$ if and only if  there exists a connected curve $C$ such that $x_1, x_2 \in C$ and $C$ is build-up from cycles represented by $V$ (see \cite[IV, Definition 4.7]{kollar96} for the formal statement). The proalgebraic relation $R$ is a countable union of proper algebraic relations $R_1, R_2, R_3, \cdots$.

\begin{theorem}[{\cite[IV, Theorem 4.16]{kollar96}}] \label{theorem:quotients} Let $m$ be any natural number. Then the quotient of a normal projective variety $X$ by $R_1, \ldots, R_m$ as above exists. More precisely, there exists an open subvariety $X^0 \subseteq X$ and a proper morphism $\pi \colon X^0 \to Z^0$ with connected fibres such that:
\begin{itemize}
	\item $R$ is an equivalence relation on $X^0$, and
	\item $\pi^{-1}(z)$ coincides with an $R$-equivalence class for every $z \in Z^0$.
\end{itemize}
\end{theorem}
When $R$ is the union of $R_1, \ldots, R_m$, this says that a curve $C \subseteq X$ intersecting $X^0$ is contracted by $\pi$ if and only if $C$ is build-up from curves represented by $V$.

 If $R$ is merely a countable union of algebraic relations, then the quotient exists as well (see \cite[IV, Theorem 4.17]{kollar96}), but $\pi^{-1}(z)$ coincides with an $R$-equivalence class only for a very general $z \in Z^0$. This is an easy consequence of the above theorem as the quotients by $R_1, \ldots, R_m$ must stabilise for $m \gg 0$.

\subsection{Nef reduction map} \label{subsection:nef_reduction_map}
This subsection is based on \cite{bcekprsw02}. The results therein are stated for complex algebraic varieties, but the proofs are valid for all varieties defined over uncountable algebraically closed fields. 

\begin{definition} \label{definition:nef_reduction_map}
Let $X$ be a normal projective variety defined over an uncountable field $k$ and let $L$ be a nef $\mbQ$-Cartier $\mbQ$-divisor. We call a rational map $\phi \colon X \dashrightarrow Z$ a \emph{nef reduction map} if $Z$ is a normal projective variety and there exists an open dense subset $V \subseteq Z$ such that
\begin{enumerate}
	\item $\phi$ is proper over $V$ and $\phi_*\mcO_{\phi^{-1}(V)} = \mcO_V$,
	\item $L|_F \equiv 0$ for all fibres $F$ of $\phi$ over $V$, and
	\item if $x \in X$ is a very general point and $C$ is a curve passing through it, then $C \cdot L = 0$ if and only if $C$ is contracted by $\phi$.
\end{enumerate}
\end{definition}

\begin{theorem}[{\cite[Theorem 2.1]{bcekprsw02}}] \label{theorem:nef_reduction_map} A nef reduction map exists for normal projective varieties defined over an uncountable algebraically closed field $k$. Further, it is unique up to a birational morphism.
\end{theorem}

We call $n(L) = \dim Z$ the \emph{nef dimension} of $L$, where $Z$ is the target of a nef reduction map $\phi \colon X \dashrightarrow Z$. 

The following result is fundamental in the proof of Theorem \ref{theorem:nef_reduction_map}. 
\begin{theorem}[{\cite[Theorem 2.4]{bcekprsw02}}] \label{theorem:numerically_trivial_fibres} Let $X$ be a projective variety, non-necessarily normal, defined over an uncountable algebraically closed field $k$. Let $L$ be a nef line bundle on $X$. Then $L$ is numerically trivial if and only if any two points in $X$ can be joined by a connected chain $C$ of curves such that $L\cdot C = 0$.
\end{theorem}

Now, let $X$ be a normal projective variety defined over a countable algebraically closed field $k$, for instance $k= \Fp$ or $k = \bar{\mbQ}$, and let $L$ be a nef line bundle on it. The definition of a nef reduction map makes no sense for such $X$ as the set of its very general points is empty. 

To circumvent this problem, we choose any uncountable algebraically closed field $K$ containing $k$, and take $X_K$, $L_K$ to be the base changes of $X$ and $L$ to $\Spec K$, respectively. Thereafter, we can define the nef dimension as $n(L) \defeq n(L_K)$. In fact, the nef reduction map of $L_K$ is defined over $k$, as the following result shows.
\begin{proposition} \label{prop:nef_reduction_map_fp} With notation as above, there exist a rational map $\phi \colon X \dashrightarrow Z$ to a normal projective variety $Z$ and an open dense subset $V \subseteq Z$ such that
\begin{enumerate}
  \item $\phi$ is proper over $V$ and $\phi_*\mcO_{\phi^{-1}(V)} = \mcO_V$,
  \item $L|_F \equiv 0$ for all fibres $F$ of $\phi$ over $V$, and
  \item if $\phi_K$ is the base change of $\phi$ to $\Spec K$, then $\phi_K$ is a nef reduction map of $L_K$.
\end{enumerate}
In particular, $\dim Z = n(L)$.
\end{proposition}
\noindent Morphisms $\phi$ as above are independent of the choice of $K$.
\begin{proof}
Let $V' \subseteq \Chow(X_K)$ be the countable union of all irreducible components of the Chow variety parametrising one-cycles $C$ satisfying $L_K \cdot C = 0$. Now, if $\phi_K \colon X_K \dashrightarrow Z'$ is any nef reduction map of $L_K$, then it must be a quotient by the algebraic equivalence relation induced by some irreducible components $V'_1, \ldots, V'_r$ of $V'$. By definition, $\Chow(X_K) \simeq \Chow(X) \times_k \Spec K$, and so there exist irreducible components $V_1, \ldots, V_r \subseteq \Chow(X)$ such that $V'_i \simeq V_i \times_k \Spec K$. Then, the assertion of the proposition is satisfied for $\phi$ being the quotient of $X$ by the algebraic equivelence relation induced by $V_1, \ldots, V_r$ (see Theorem \ref{theorem:quotients}). 
\end{proof}

\begin{definition} \label{definition:nef_reduction_map_fp} We call a morphism satisfying the assertions of Proposition \ref{prop:nef_reduction_map_fp} a \emph{nef reduction map} of $L$ on $X$.
\end{definition}

It is natural to ask if a nef line bundle descends under its nef reduction map. This is known to hold when the restriction of $L$ to the generic fibre is torsion, or when $n(L)=1$.

\begin{lemma} \label{lemma:descending_nef_line_bundles} Let $f \colon X \to Z$ be an equidimensional proper morphism between normal projective varieties. Let $L$ be a nef $\mbQ$-Cartier divisor on $X$ such that $L|_{X_{\mu}} \sim_{\mbQ} 0$, where $X_{\mu}$ is the generic fibre. Then, there exists a $\mbQ$-divisor $L_Z$ on $Z$ such that $L \sim_{\mbQ} f^*L_Z$.
\end{lemma}
\begin{proof}
By assumptions, there exists a $\mbQ$-divisor $L_{Z}$ on $Z$ and an effective $\mbQ$-divisor $E$ on $Z$ such that:
\begin{itemize}
	\item $L + E \sim_{\mbQ} f^*L_Z$,
	\item $\dim f(E) < \dim Z$, and
	\item for any irreducible component $V \subseteq f(E)$, we have $f^{-1} V \not \subseteq \Supp E$.
\end{itemize}
By Zariski's lemma (more precisely \cite[Lemma 1.5]{fujita86b}), $E = 0$, which concludes the proof.
\end{proof}
Note that $f^*L_Z$ is well defined even when $L_Z$ is not $\mbQ$-Cartier, because $f$ has equidimensional fibres.

\begin{lemma}\label{lemma:descending_nef_line_bundles_to_curves}  Let $X$ be a normal projective three-dimensional variety defined over an algebraically closed field $k$, let $C$ be a smooth curve, let $f \colon X \to C$ be a contraction, and let $X_{\mu}$ be the generic fibre. Further, let $L$ be a nef $\mbQ$-Cartier divisor on $X$ such that $L|_{X_{\mu}} \equiv 0$. Then there exists a $\mbQ$-divisor $D$ on $C$ such that $L \equiv f^*D$.
\end{lemma}  
\begin{proof}
By \cite[Lemma 5.3]{bw14}, we have $L \equiv_f 0$. Replacing $X$ by a resolution of singularities, we can assume that $X$ is smooth. Thus we can conclude the proof by \cite[Lemma 5.2]{bw14}.
\end{proof}

Last, let us state the following result which allows for applying the nef reduction map in the context of the base point free theorem.
\begin{theorem}[{\cite[Theorem 1.4]{ctx13}}] \label{theorem:birkar} Let $(X,\Delta)$ be a projective log pair defined over an algebraically closed field $k$ of characteristic $p>0$, and let $L$ be a nef $\mbQ$-Cartier divisor such that  $L-(K_X+\Delta)$ is nef and big and $n(L)=\dim X$. Then $L$ is big.
\end{theorem}
\section{A weak canonical bundle formula} \label{s:cbf}
The goal of this section is to show Theorem \ref{theorem:canonical_bundle_formula_main} (cf.\ Theorem \ref{theorem:canonical_bundle_formula_main_intro}) and Theorem \ref{theorem:canonical_bundle_formula_3fold} (cf.\ Theorem \ref{theorem:canonical_bundle_formula_3fold_intro}). 

\subsection{The canonical bundle formula for boundaries separable over the base}
In this subsection, we derive the weak version of the canonical bundle formula for when $(X_{\overline \mu}, \Delta|_{X_{\overline \mu}})$ is log canonical where $X_{\overline \mu}$ is the geometric generic fibre of $\phi$ (Proposition \ref{proposition:canonical_bundle_formula}). To this end, we need the following result. 

\begin{lemma}\label{lem:cbf_standard}
Let $(X,\Delta)$ be a quasi-projective log pair defined over an algebraically closed field $k$ of characteristic $p>0$ and let $\phi \colon X \to Z$ be a contraction. Assume that the generic fibre $X_{\mu}$ of $\phi$ is a smooth rational curve, the coefficients of $\Delta^h$ are at most one, and $K_X + \Delta \sim_{\mbQ} \phi^*L_Z$ for some $\mbQ$-Cartier $\mbQ$-divisor $L_Z$ on $Z$. Further, suppose that $\Supp \Delta^h$ is a union of sections of $\phi$ over the generic point of $Z$. Then $L_Z \sim_{\mbQ} K_Z + \Delta_Z$ for some effective $\mbQ$-divisor $\Delta_Z$.
\end{lemma}
\noindent 
This lemma follows from the standard canonical bundle formula with $\Supp \Delta^h$ separable and tamely ramified over $Z$, a result which is verified in \cite[Subsection 6.2]{ctx13} (see also \cite[Theorem 4.8]{hacondas} and \cite{ps09}). Thus, we only provide a sketch of the argument.
\begin{proof}
We are going to show that $\Delta_Z = \Delta_{\mathrm{div}} + \Delta_{\mathrm{mob}}$, where $\Delta_{\mathrm{mob}}$ is an effective $\mbQ$-divisor and
\begin{align*}
\Delta_{\mathrm{div}} &\defeq \sum_{\text{Weil divisor } Q} (1 - c_Q)Q, \text{ for}\\
c_Q &\defeq \sup\,\{c \in \mbQ \mid (X,\Delta + c\phi^*Q) \text{ is lc over the generic point of } Q\}.
\end{align*}
As $c_Q \leq 1$, we get that $\Delta_{\mathrm{div}}$ is effective.

Consider the following diagram:
\begin{center}
\begin{tikzcd}
& & \overline{X} \arrow[swap]{ld}{\pi_2} \arrow{d}{\pi_1}   \\ 
\overline{U}_{0,m} \arrow{d}{\gamma} & U \arrow[dashed]{r}\arrow[swap]{dr}{\phi_U} \arrow[swap]{l}{\psi} & X \arrow{d}{\phi}  \\
\overline{M}_{0,m} & &  Z \arrow{ll}{u} ,
\end{tikzcd}
\end{center}
where 
\begin{itemize}
	\item $\overline{M}_{0,m}$ is the moduli space of $m$-pointed stable curves of genus zero,
	\item $\gamma \colon \overline{U}_{0,m} \to \overline{M}_{0,m}$ is a $\mbP^1$-bundle constructed by blowing-down divisors on the universal family (see \cite[Lemma 4.6]{hacondas}),
	\item $u \colon Z \dashrightarrow \overline{M}_{0,m}$ is a map given by the fact that $\phi \colon X \to Z$ is generically a $\mathbb{P}^1$-bundle with $m$ sections $\Supp \Delta^h$,
	\item $U \defeq \overline{U}_{0,m} \times_{\overline{M}_{0,m}} Z$,
	\item $\pi_1 \colon \overline{X} \to X$ is a birational resolution of the rational map $U \dashrightarrow X$. 
\end{itemize}
Since effective divisors extend in codimension one, we can replace $Z$ by a smooth open subset $Z' \subseteq Z$ whose complement is of codimension two and $X$ by $X \times_Z Z'$ so that $u \colon Z \to \overline{M}_{0,m}$ is a well defined morphism. Since $Z$ is smooth, $K_Z+\Delta_{\mathrm{div}}$ is a $\mbQ$-Cartier $\mbQ$-divisor.

To prove the lemma, it is enough to show that $K_X+\Delta - \phi^*(K_Z + \Delta_{\mathrm{div}})$ is semiample. 
Define $\Delta_{\overline X}$  as the log pullback of $\Delta$ via $\pi_1$, and set \[
\Delta_U \defeq (\pi_2)_*\Delta_{\overline{X}}.
\] By construction of $\overline{U}_{0,m}$ and by the two-dimensional inversion of adjunction, we have that $(U, \Delta_U^h + \phi_U^{-1}Q)$ are lc over the generic points of irreducible divisors $Q\subseteq \Supp \Delta_{\mathrm{div}}$ (cf.\ \cite[(4.11)]{hacondas}). Since $\phi_U$ is a $\mathbb{P}^1$-fibration, this implies that $\Delta_U^v = \phi_U^* \Delta_{\mathrm{div}}$. Thus
\[
K_{U} + \Delta_U - (\phi_U)^*(K_Z + \Delta_{\mathrm{div}}) \sim_{\mbQ} \psi^*(K_{\overline{U}_{0,m}/\overline{M}_{0,m}} + D) 
\]
where $D$ is a sum of the canonical sections of $\gamma$ with rational coefficients in $[0,1]$ and $K_{\overline{U}_{0,m}/\overline{M}_{0,m}} + D$ is semiample by \cite[Lemma 4.6]{hacondas}. This finishes the proof of the lemma. 
\end{proof}

The following result provides the first step in the proof of Theorem \ref{theorem:canonical_bundle_formula_main}.
\begin{proposition}\label{proposition:canonical_bundle_formula} Let $(X,\Delta)$ be a quasi-projective log pair defined over an algebraically closed field $k$ of characteristic $p>0$ and let $\phi \colon X \to Z$ be a contraction. Assume that the geometric generic fibre $X_{\bar{\mu}}$ of $\phi$ is a smooth curve, $(X_{\bar{\mu}}, \Delta|_{X_{\bar{\mu}}})$ is log canonical, and $K_X + \Delta \sim_{\mbQ} \phi^*L_Z$ for some $\mbQ$-Cartier $\mbQ$-divisor $L_Z$ on $Z$. Then $L_Z \sim_{\mbQ} K_Z + \Delta_Z$ for an effective $\mbQ$-divisor $\Delta_Z$. 
\end{proposition} 
If we only require $\Delta_Z$ to be pseudo-effective, then the proposition follows from \cite[Theorem 4.5]{ejiri16a} (cf.\ \cite[Corollary 4.11]{ejiri16a}). Note that in characteristic zero, if $(X,\Delta)$ is log canonical, then so is $(X_{\bar{\mu}}, \Delta|_{X_{\bar{\mu}}})$, but this need not be true in positive characteristic due to the existence of inseparable morphisms.
\begin{proof}
By blowing up along components of $\Sing(X)$ of codimension two in $X$ which are not contracted by $\phi$, we can obtain a birational morphism  $g \colon Y \to X$ such that $g^*(K_X+\Delta) = K_Y + \Delta_Y$ for a $\mbQ$-divisor $\Delta_Y$ satisfying
\[
\codim_Z\left(\phi \circ g\left( \Sing Y \cup \Supp \left(\Delta_Y\right)_{<0}\right)\right) \geq 2,
\]
where $(\Delta_Y)_{<0}$ is the negative part of $\Delta_Y$. Indeed, consider the localisations $X_{\zeta_i}$ of $X$ at the generic points of the components of $\Sing(X)$ described above. Since $X_{\zeta_i}$ is of dimension two, we can construct its minimal resolution by a series of blow-ups. Extending the blown-up loci to $X$ and taking normalisation provides us with $Y$ as above. We replace $X$ by $Y$ and $\Delta$ by $\Delta_Y$. Further, we replace $Z$ by a smooth open subscheme $U \subseteq Z$ whose complement is of codimension two, and $X$ by $\phi^{-1}(U)$. Thus, we can assume that $\phi$ is flat, $\Delta \geq 0$, and $X$, $Z$ are smooth.

Since $(X_{\overline \mu}, \Delta|_{\overline \mu})$ is log Calabi-Yau, $X_{\overline \mu}$ is either a smooth rational curve or an elliptic curve. In the latter case, $\kappa(K_{X/Z})\geq 0$ by \cite[Subsection 3.1 and Claim 3.1]{cz15}. Thus, we can write $K_X + \Delta = \phi^*K_Z + E + \Delta$ for some effective $\mbQ$-divisor $E \sim_{\mbQ} K_{X/Z}$, and by the Zariski lemma (see \cite[Lemma 1.5]{fujita86b}), we get that $E + \Delta = \phi^*\Delta_Z$ for some effective $\mbQ$-divisor $\Delta_Z$ on $Z$. This finishes the proof when $X_{\overline \mu}$ is an elliptic curve.

As of now, assume that $X_{\overline \mu} \simeq \mbP^1$, and consider the following Cartesian diagram:
\begin{center}
\begin{tikzcd}
X_T \arrow{r}{\psi} \arrow{d}{\phi_T} & X \arrow{d}{\phi} \\
T \arrow{r}{f} & Z,
\end{tikzcd}
\end{center}
where $f \colon T \to Z$ is any finite morphism with $T$ normal which generically factors through $\phi|_{D_i} \colon D_i \to Z$ for every irreducible divisor $D_i \subseteq \Supp \Delta^h$ (for instance, set $T$ to be the normal closure of $Z$ in all the fraction fields $K(D_i)$), and $X_T$ is the normalisation of $X \times_Z T$. By replacing $Z$ by a smaller open subscheme with complement of codimension two, we can assume that $T$ is smooth. Then Proposition \ref{proposition:relative_canonical_divisor_under_base_change} implies that
\[
\psi^*(K_X+\Delta) \sim_{\mbQ} K_{X_T/T} + \Delta_T + \phi_T^*f^*K_Z
\]
for some effective $\mbQ$-divisor $\Delta_T$. Since the condition that $(X_{\bar{\mu}}, \Delta|_{\bar{\mu}})$ is log canonical is stable under base change, the coefficients of the horizontal part $\Delta_T^h$ are at most one. Moreover $\Supp \Delta_T^h$ is a union of sections of $\phi_T$ over the generic point of $T$. Thus by Lemma \ref{lem:cbf_standard} we have
\[
\psi^*(K_X+\Delta) \sim_{\mbQ}  \phi_T^*(f^*K_Z + E),
\]
where $E$ is some effective $\mbQ$-divisor on $T$. Therefore,
\[
L_Z \sim_{\mbQ} K_Z + \frac{1}{\deg f}f_*E. \qedhere
\]

\end{proof}
%

\subsection{The proof of the weak canonical bundle formula}

In the proof of Theorem \ref{theorem:canonical_bundle_formula_main}, we shall need the following lemma.

\begin{lemma}[Purely inseparable relative adjunction] \label{lemma:inseparable_adjunction} Let $X$ be a regular quasi-projective variety defined over a field $k$ of characteristic $p>0$, let $\phi \colon X \to Z$ be a contraction whose generic fibre is a smooth curve, let $S \subseteq X$ be a divisor such that $\phi|_S \colon S \to Z$ is purely inseparable of degree $p^k$, and let $\tilde{S}$ be the normalisation of $S$. Then,
\[
\big(K_{X/Z} + \frac{1}{p^k}S\big)|_{\tilde S} \sim_{\mbQ} \Delta,
\]
for some effective $\mbQ$-divisor $\Delta$ on $\tilde S$.
\end{lemma}
In fact, the lemma holds if $X$ is only $\mbQ$-Gorenstein in codimension two (see the beginning of the proof of Proposition \ref{proposition:canonical_bundle_formula}). By means of localisation at codimension two points, one can see that all klt varieties satisfy this condition.

\begin{proof}
We can assume that $Z$ is regular and $\phi$ is flat by replacing $Z$ by an open regular subset $U \subseteq Z$ whose complement is of codimension two, and $X$ by $\phi^{-1}(U)$. To ease the notation, write $T \defeq \tilde S$ and $f \defeq \phi|_{\tilde S}$. 

Consider the following diagram
\begin{center}
\begin{tikzcd}
 Y \arrow{r}{g} \arrow{d}{\phi'} & X \arrow{d}{\phi} \\
T \arrow{r}{f} & Z,
\end{tikzcd}
\end{center}   
where $Y$ is the normalisation of $X \times_Z T$. The construction yields a natural induced-by-$f$ section $T \to Y$ whose image by abuse of notation we denote by $T$. 

 By Proposition \ref{proposition:relative_canonical_divisor_under_base_change}, there exists an effective $\phi'$-vertical divisor $D$ such that
\[
g^*(K_{X/Z} + \frac{1}{p^k}S)= K_{Y/T} + T + D. 
\]
The lemma follows by the standard adjunction on $T$. 
\end{proof}

We are ready to prove the main theorem of this section.
\begin{theorem}[{Theorem \ref{theorem:canonical_bundle_formula_main_intro}}] \label{theorem:canonical_bundle_formula_main} Let $(X,\Delta)$ be a quasi-projective log pair defined over an algebraically closed field $k$ of characteristic $p>0$ and let $\phi \colon X \to Z$ be a contraction with the generic fibre being a smooth curve. Assume that the $\phi$-horizontal divisors in $\Delta$ have coefficients at most one, and $K_X+\Delta \sim_{\mbQ} \phi^* L_Z$ for some $\mbQ$-Cartier $\mbQ$-divisor $L_Z$ on $Z$.

Then, there exist a finite purely-inseparable morphism $f \colon T \to Z$ and an effective $\mbQ$-divisor $E$ on $T$, such that
\[
f^*L_Z \sim_{\mbQ} tK_{T}  + (1-t)f^*K_Z + E,
\]
for some rational number $0 \leq t \leq 1$.

Furthermore, if $\Delta = B + A$ for an effective $\mbQ$-divisor $B$ and an ample $\mbQ$-Cartier $\mbQ$-divisor $A$, then $E$ is big.
\end{theorem}
If $p>3$, or $\Delta$ is big and $p>2$, then the generic fibre is automatically smooth by Proposition \ref{proposition:general_fibre_smooth}.
\begin{proof}
By the same argument as in Proposition \ref{proposition:canonical_bundle_formula} we can blow up along components of $\Sing(X)$ of codimension two in $X$ which are not contracted by $\phi$, and obtain a birational morphism  $g \colon Y \to X$ such that $g^*(K_X+\Delta) = K_Y + \Delta_Y$ for a $\mbQ$-divisor $\Delta_Y$ satisfying
\[
\codim_Z\left(\phi \circ g\left( \Sing Y \cup \Supp \left(\Delta_Y\right)_{<0}\right)\right) \geq 2,
\]
where $(\Delta_Y)_{<0}$ is the negative part of $\Delta_Y$. 
We replace $X$ by $Y$ and $\Delta$ by $\Delta_Y$ so that we have $\codim_Z(\phi\,(\Sing X \cup \Supp \Delta_{<0}))\geq 2$. 

Let $L \defeq K_X+\Delta$ and let $F$ be a general fibre of $\phi$. Note that since $(K_X+\Delta)\cdot F=0$, we have $\Delta \cdot F \leq 2$. If $(F,\Delta|_{F})$ is log canonical, then the theorem follows from Proposition \ref{proposition:canonical_bundle_formula} with $f=\mathrm{id}$. Thus, we can assume that it is not log canonical, and so we can write $\Delta = aS + B$ for $0 < a \leq 1$, an irreducible divisor $S$, and an effective $\mbQ$-divisor $B$ such that $S \not \subseteq \Supp B$ and
\[
(F, aS|_{F})
\]
is not log canonical. We claim that $\phi|_S$ is generically purely inseparable. To this end, write $S|_F = \sum_{i=1}^r p^kP_i$ for some distinct points $P_i \in F$, where $p^k$ is the degree of the purely inseparable part of the extension $K(S)/K(Z)$. Since $aS \cdot F \leq 2$, we have $arp^k\leq 2$, and since $(F, aS|_F)$ is not log canonical, $ap^k > 1$ holds. Therefore, $r=1$ and $\phi|_S$ is generically purely inseparable of degree $p^k>1$. Further, $\frac{1}{p^k} < a \leq \frac{2}{p^k}$.

Let $\tilde{S}$ be the normalisation of $S$  and let
\begin{center}
\begin{tikzcd}
\tilde{S} \arrow{r}{g} \arrow[bend left = 35]{rr}{h} & T \arrow{r}{f} & Z.
\end{tikzcd}
\end{center}
be the Stein factorisation of $h \defeq \phi|_{\tilde S}$, wherein $g$ is birational. Since divisors uniquely extend in codimension one and $\codim_Z(\phi\,(\Sing X \cup \Delta_{<0}))\geq 2$, we can assume that $g$ is an isomorphism, $\Delta \geq 0$ and $X$, $Z$ are smooth by replacing $Z$ by an open smooth subset $U \subseteq Z$ whose complement is of codimension two and $X$, $S$, $\tilde{S}$ by $\phi^{-1}(U)$, $\phi^{-1}(U) \cap S$, and $h^{-1}(U)$, respectively. 

Recall that $\Delta = aS + B$, take $0 \leq t < 1$ to be a rational number such that $a = \frac{1}{p^k}t + (1-t)$, and write
\[
L = t(K_X + \frac{1}{p^k}S) + (1-t)(K_X + S) + B.
\]
By Lemma \ref{lemma:inseparable_adjunction} and the standard adjunction we get
\[
h^*L_Z \sim_{\mbQ} L|_{\tilde S} \sim_{\mbQ} th^*K_Z + (1-t)K_{\tilde S} +  E_{\tilde S}
\]
for some effective $\mbQ$-divisor $E_{\tilde S}$. Since $g$ is an isomorphism, this finishes the proof of the main part of the theorem. 

Last, if $\Delta = B + A$ as in the statement of the theorem, then $A \sim_{\mbQ} A' + \phi^*A_Z$ for some sufficiently general ample $\mbQ$-Cartier $\mbQ$-divisors $A'$ and $A_Z$ on $X$ and $Z$, respectively. By applying the above proof to $\Delta$ replaced by $B + A'$, we can conclude.
\end{proof}

\begin{example} \label{ex:cbf_false} It is easy to see that the canonical bundle formula is false in positive characteristic even for smooth $\mbP^1$-fibrations from surfaces to curves (cf.\ \cite[V, Exercise 2.15]{hartshorne77}). Here, we give an example that in Theorem \ref{theorem:canonical_bundle_formula_main} it is moreover not enough to take $f$ to be a power of Frobenius. 

Let $S$ be a smooth projective surface defined over an algebraically closed field of characteristic two such that $K_S$ is ample, $H^1(S, \omega_S^{-1}) \neq 0$, but $H^1(S, \omega_S^{-2}) =0$. Such a surface exists by \cite[Proposition I.2.14 and Theorem II.1.7]{ekedahl88}. Let $\mcE$ be a vector bundle constructed as a non-trivial extension of $\omega_S$ by $\mcO_S$ and let $X \defeq \mbP_S(\mcE)$. By assumptions, $F^*\mcE$ splits providing a divisor $D$ on $X$ such that $\pi|_D$ is purely inseparable of degree two, where $\pi \colon X \to S$ is the natural projection.

We claim that $L \defeq K_X + D$ is numerically trivial. Since $S$ is of general type, the claim shows that $L$ cannot descend to a $\mbQ$-divisor of the form $\lambda K_S + P$, where $\lambda>0$ and $P$ is pseudo-effective. For the assertion of Theorem \ref{theorem:canonical_bundle_formula_main} to hold one takes $f$ to be a purely-inseparable cover of $S$ by $D$.

In order to show that $L$ is numerically trivial, we consider the following Cartesian diagram
\begin{center}
\begin{tikzcd}
X' \arrow{r}{\phi} \arrow{d}{\rho} & X \arrow{d}{\pi} \\
S \arrow{r}{F} & S,
\end{tikzcd}
\end{center}
where $X' \defeq \mbP_S(F^*\mcE)$. Since $\pi$ is smooth, so is $\rho$. In particular,
\[
\phi^*L = \phi^*K_{X/S} + \rho^*F^*K_S + 2S = K_{X'} + \rho^*K_S + 2S,
\]
wherein $S \subseteq X'$ is a section of $\rho$ such that $S|_S \sim -2K_S$. Thus,
\[
\phi^*L|_{S} \sim 0,
\]  
and since $\pi$ is a $\mbP^1$-fibration and $L$ is $\pi$-numerically trivial, it must be numerically trivial.
\end{example}

\begin{remark} Theorem \ref{theorem:canonical_bundle_formula_main} and Proposition \ref{proposition:canonical_bundle_formula} should work over any field $k$ of positive characteristic $p>0$ (non necessarily algebraically closed) contingent upon the verification that the results in \cite[Subsection 3.1 and Claim 3.1]{cz15} are valid over $k$.
\end{remark}

\subsection{Low dimensions}
In this subsection, we show Theorem~\ref{theorem:canonical_bundle_formula_3fold} (cf.\ Theorem~\ref{theorem:canonical_bundle_formula_3fold_intro}). The following lemma is a key component in its proof.
\begin{lemma} \label{lemma:pseudoeffective_surfaces} Let $X$ be a normal projective $n$-dimensional variety defined over an algebraically closed  field $k$ of characteristic $p>0$ and let $L$ be a nef $\mbQ$-Cartier $\mbQ$-divisor on $X$ such that $n(L)=\dim X$. Assume that $n =2$, or $p > 5$ and $n=3$. Then, there exists a rational number $\lambda > 0$ for which $K_X  + \lambda L$ is pseudo-effective.
\end{lemma}
The lemma is not true when $n(L) < \dim X$. We refer to \cite{tanaka12} for the main results of the Minimal Model Program for surfaces in positive characteristic, and to \cite{hx13}, \cite{ctx13}, \cite{birkar13}, and \cite{bw14} in the case of threefolds.

\begin{proof}
Let $\pi \colon \overline{X} \to X$ be a resolution of singularities. Assume that there exists $\lambda > 0$ such that $K_{\overline{X}} + \lambda\pi^*L$ is pseudo-effective. Then, $K_X + \lambda L$ is pseudo-effective as well. Thus, by replacing $X$ by $\overline{X}$, we can assume that $X$ is terminal and $\mbQ$-factorial. We do not assume that $X$ is smooth for inductive reasons. 

First, we show the lemma under an additional assumption that all $K_X$-negative extremal rays $R$ satisfy $L\cdot R > 0$. To this end, set $\lambda = 2nm$ for the Cartier index $m \in \mbN$ of $L$. It is enough to show that $K_X + \lambda L$ is nef. By contradiction, assume it is not. Then, there exists a curve $C$ such that $(K_X + \lambda L) \cdot C < 0$. Since $L \cdot C \geq 0$, we must have $K_X \cdot C < 0$. By the cone theorem, $C$ is an affine combination of $K_X$-negative extremal rays. In particular, there must exist a $K_X$-negative extremal curve $R$ satisfying $(K_X+\lambda L) \cdot R < 0$. However, by the additional assumption, $mL \cdot R \geq 1$, hence
\[
(K_X + \lambda L) \cdot R  \geq 0,
\]
since $R$ belongs to an extremal ray and so it satisfies $R \cdot K_X \geq -2n$. This is a contradiction. 

In general we run a $K_X$-MMP which contracts only those rays $R$ which satisfy $L \cdot R = 0$. Since $n(L)=\dim X$ and $X$ has terminal singularities this MMP terminates with a birational model (cf.\ \cite{hx13}). Let $f \colon X \dashrightarrow Y$ be the first step of the MMP and let $L_Y \defeq f_* L$. By induction (the base of the induction is satisfied by the above paragraph), we can assume that $P_Y \defeq K_Y + \lambda L_Y$ is pseudo-effective for some $\lambda>0$.  If $f$ is a flip, then $K_X + \lambda L = f^*P_Y$ is pseudo-effective as well. If $f$ is a contraction of a divisor $D$, then
\[
K_X + \lambda L \sim_{\mbQ} f^*P_Y + aD
\]
for some $a \in \mbQ$. Since $K_X\cdot R < 0$, we have $a > 0$. This concludes the proof of the lemma.
\end{proof}

\begin{theorem}[{Theorem \ref{theorem:canonical_bundle_formula_3fold_intro}}] \label{theorem:canonical_bundle_formula_3fold} Let $(X,\Delta)$ be a projective $n$-dimensional log pair defined over an algebraically closed field $k$ of characteristic $p>0$ and let $\phi \colon X \to Z$ be a contraction with the generic fibre being smooth and of dimension one. Suppose that the $\phi$-horizontal divisors in $\Delta$ have coefficients at most one, $K_X + \Delta$ is nef, and $K_X+\Delta \sim_{\mbQ} \phi^* L_Z$ for some $\mbQ$-Cartier $\mbQ$-divisor $L_Z$ on $Z$. Further, suppose that $n(K_X+\Delta)=n-1$ or more generally that $K_Z+\lambda L_Z$ is pseudo-effective for some $\lambda>0$. 

Assume that $n\leq 3$, or $p>5$ and $n=4$. Then, there exists a finite purely-inseparable morphism $f \colon T \to Z$, with $T$ normal, and a pseudo-effective $\mbQ$-divisor $\Delta_{T}$ on $T$ such that 
\[
f^*L_Z \sim_{\mbQ} K_{T} + \Delta_{T}. 
\]
Furthermore, if $\Delta = B + A$ for an effective $\mbQ$-divisor $B$ and an ample $\mbQ$-Cartier $\mbQ$-divisor $A$, then $\Delta_T$ is big.
\end{theorem}
We do not know if the assumption that $n(K_X+\Delta)=\dim Z$ (or more generally that $K_Z+\lambda L_Z$ is pseudo-effective) is necessary.
\begin{proof}
By Theorem \ref{theorem:canonical_bundle_formula_main}, there exists a purely inseparable morphism $f \colon T \to Z$ such that
\[
L_T \defeq f^*L_Z \sim_{\mbQ} tf^*K_Z + (1-t)K_T + E
\]
for some effective $\mbQ$-divisor $E$ on $T$ and $0 \leq t \leq 1$. If $t=1$, then $L_Z \sim_{\mbQ} K_Z + \frac{1}{\deg f}f_*E$. Thus, we can assume that $t<1$. Lemma \ref{lemma:pseudoeffective_surfaces} implies that there exists $\lambda >0$ such that $P \defeq K_Z + \lambda L_Z $ is pseudo-effective. Therefore,
\begin{align*}
(1+\lambda)L_T &\sim_{\mbQ} tf^*(K_Z+\lambda L_Z) + (1-t)K_T + E \\
&=\ \, (1-t)K_T + tf^*P + E.
\end{align*}
Thus, there exists a rational number $b > 0$ and a pseudo-effective $\mbQ$-divisor $Q$ such that 
\[
bL_T \sim_{\mbQ} K_T + Q.
\]
By replacing $f$ by its composition with some power of Frobenius, we may assume that
\[
\frac{b}{p^k}L_T \sim_{\mbQ}K_T + Q
\]
for $k \geq 0$ such that $\frac{b}{p^k} < 1$. Then
\[
L_T \sim_{\mbQ} K_T + Q + (1-\frac{b}{p^k})L_T,
\]
which concludes the proof.

\end{proof}

\begin{remark} \label{rem:cbf_3fold} Theorem \ref{theorem:canonical_bundle_formula_3fold} is valid for contractions of relative dimension one from a normal projective variety of dimension $n$ contingent upon the validity of the terminal Minimal Model Program in dimension $n-1$ (including the termination of arbitrary sequences of flips) as well as the existence of resolutions of singularities.
\end{remark}


\section{Applications} \label{s:bpf}
The goal of this section is to show Theorem \ref{thmi:non_vanishing}, \ref{thmi:abundance}, \ref{thmi:bpf}, and \ref{thmi:nonvanishing_bpf}. 

\subsection{Nef reduction map and the canonical bundle formula}
In this subsection, we prepare the toolkit for tackling the case of nef dimension two in the applications of the canonical bundle formula to threefolds; more precisely we deal with the issue that nef reduction maps are well defined only up to a birational modification. On the first reading of the article, the reader might consider skipping this subsection by assuming that the encountered nef reduction maps are proper morphisms. With this assumption, Proposition \ref{prop:nef_reduction_cbf} is a special case of Theorem \ref{theorem:canonical_bundle_formula_3fold}.

We start by stating a simple lemma used to resolve rational maps.
\begin{lemma} \label{lemma:nef_reduction_map_3folds} Let $X$ be a normal projective threefold defined over an algebraically closed field $k$, and let $L$ be a nef Cartier divisor on it. Let $\phi \colon X \dashrightarrow Z$ be a map of relative dimension one which is proper over an open subset $V \subseteq Z$  and such that $\phi_* \mcO_{\phi^{-1}(V)} = \mcO_V$. Assume that $L|_{X_{\mu}} \sim_{\mbQ} 0$ where $X_{\mu}$ is the generic fibre. Then the map $\phi$, after possibly replacing $Z$ by a birational modification, admits a birational resolution
\begin{center}
\begin{tikzcd}
Y \arrow[swap]{r}{g} \arrow[bend left = 25]{rr}{\psi} & X \arrow[dashed, swap]{r}{\phi} & Z.
\end{tikzcd}
\end{center} 
with $\psi$ equidimensional, $Y$ normal, $Z$ smooth, and such that
\begin{enumerate}
	\item $g^*L \sim_{\mbQ} \psi^*L_Z$ for a $\mbQ$-divisor $L_Z$ on $Z$, and
	\item no curve on $Y$ is contracted by both $g$ and $\psi$.
\end{enumerate}
\end{lemma}
In particular, (2) implies that no $g$-exceptional divisor on $Y$ is contracted by $\psi$ to a point. By a birational modification of $Z$, we mean an arbitrary surface $Z'$ which is birational to $Z$. 
\begin{proof}
First, we pick an arbitrary resolution of $\phi$, that is a birational morphism $g \colon Y \to X$ for which there exists a map $\psi \colon Y \to Z$ such that $\psi = \phi \circ g$. Replacing $\psi$ by Raynaud's flatification (see Theorem \ref{thm:equidimensional}) we can assume that $\psi$ is equidimensional. Replacing $Z$ by its resolution of singularities, and $Y$ by the normalisation of an appropriate base-change, we can assume that $Z$ is smooth. By Lemma \ref{lemma:descending_nef_line_bundles}, there exists a $\mbQ$-divisor $L_Z$ on $Z$ such that $g^*L \sim_{\mbQ} \psi^* L_Z$. Hence, (1) holds.

We show that (2) holds if we replace $Y$ by the common factorisation of $g$ and $\psi$. More precisely, we pick very ample Cartier divisors $A_X$, $A_Z$ on $X$ and $Z$, respectively, and replace $Y$ by the image of the semiample fibration associated to the basepoint-free divisor $g^*A_X + \psi^*A_Z$. Therefore, we can assume that $g^*A_X + \psi^*A_Z$ is ample, and so (2) holds. 
\end{proof}

\begin{proposition} \label{prop:nef_reduction_cbf} Let $(X,\Delta)$ be a projective three-dimensional log pair defined over an algebraically closed field $k$ of characteristic $p>3$ (or $p>2$ if $\Delta$ is big) such that $L \defeq K_X+\Delta$ is nef, $n(L)=2$, and the coefficients of $\Delta$ are at most one. Then there exists a diagram
\begin{center}
\begin{tikzcd}
& X \arrow[dashed]{d}{\phi} & Y \arrow[swap]{l}{g} \arrow{ld}{\psi} \\
T \arrow{r}{f} & Z & 
\end{tikzcd}
\end{center} 
with $g$ being birational, $\psi$ being an equidimensional contraction, $f$ being generically purely inseparable, $\phi$ being a nef reduction map for $L$, the variety $Y$ being normal, and $Z$, $T$ being smooth such that
\begin{enumerate}
	\item $g^*L \sim_{\mbQ} \psi^*L_Z$ for a $\mbQ$-divisor $L_Z$ on $Z$, and
	\item $f^*L_Z \sim_{\mbQ} K_T + \Delta_T - M$ for a pseudo-effective $\mbQ$-divisor $\Delta_T$ and an effective $\mbQ$-divisor $M$ satisfying $f^*L_Z|_{\Supp M} \equiv 0$.
\end{enumerate}
\end{proposition}
Moreover, if $\Delta$ is big and $\lfloor \Delta \rfloor = 0$, then $\Delta_T$ is big.
\begin{proof}
We shall show the proposition for $T$ being normal. Thereafter, we replace it by its minimal resolution of singularities. 

Take $\phi$ to be a nef reduction map for $L$ and note that by abundance for curves $L|_{X_{\mu}} \sim_{\mbQ} 0$ where $X_{\mu}$ is the generic fibre. Furthermore, $X_{\mu}$ is a smooth curve by Proposition \ref{proposition:general_fibre_smooth}. After possibly replacing $Z$ by a birational modification, Lemma \ref{lemma:nef_reduction_map_3folds} implies the existence of the diagram as above with $Y$ normal, $Z$  smooth, and $\psi$ a contraction. Moreover, there exists a $\mbQ$-divisor $L_Z$ on $Z$ such that $\psi^*L_Z \sim_{\mbQ} g^*L$. For the divisor $M \defeq \psi(\Exc g)$, we have $L_Z|_M \equiv 0$ and $\Exc g \subseteq \Supp \psi^*M$ (see Lemma \ref{lemma:nef_reduction_map_3folds}(2)). Write $K_Y + \Delta_Y = g^*(K_X +\Delta)$.

 Pick a rational number $b>0$ such that $\Delta_Y + b\psi^*M$ is effective. Then by Theorem \ref{theorem:canonical_bundle_formula_main} applied to $g^*L + b\psi^*M = K_Y+\Delta_Y+b\psi^*M$, there exists a purely inseparable cover $f \colon T \to Z$, a pseudo-effective $\mbQ$-divisor $\Delta_T$, and a rational number $0 \leq t \leq 1$ such that
\begin{align} \label{eq:nef_reduction_cbf}
f^*(L_Z + bM) \sim_{\mbQ} tf^*K_Z + (1-t)K_T + \Delta_T.
\end{align}

If $t=1$, then in fact we can take $f=\mathrm{id}$. If $t\neq 1$, then
\[
(1+\lambda)f^*L_Z \sim_{\mbQ} (1-t)K_T + tf^*P_Z + \Delta_T - bf^*M,
\]
where $\lambda >0$ is such that $P_Z \defeq K_Z + \lambda L_Z$ is pseudo-effective (see Lemma \ref{lemma:pseudoeffective_surfaces}). We can conclude by composing $f$ with a power of Frobenius as in the proof of Theorem \ref{theorem:canonical_bundle_formula_3fold}.

Now assume that $\Delta$ is big and $\lfloor \Delta \rfloor = 0$. We can write $\Delta \sim_{\mbQ} (1-\epsilon)\Delta + \epsilon E + \epsilon A$ where $0 < \epsilon \ll 1$ and $\Delta \sim_{\mbQ} A + E$ for an ample $\mbQ$-divisor $A$ and an effective $E$. Then Theorem \ref{theorem:canonical_bundle_formula_main} implies that $\Delta_T$ is big.
\end{proof}

\begin{remark} \label{remark:cbf_pseudoeffective}
The above proposition holds for $L = K_X + \Delta + N$, where $N$ is any nef and big $\mbQ$-Cartier $\mbQ$-divisor, by means of perturbation. Indeed, it is enough to show that Equation (\ref{eq:nef_reduction_cbf}) holds in the above proof. To this end, we invoke Theorem \ref{theorem:canonical_bundle_formula_main} as above but with $g^*L + b\psi^*M =K_Y+\Delta_Y + g^*N +b\psi^*M$ replaced by \[ K_Y + \Delta_Y + (1-\epsilon)g^*N+ \epsilon A + b\psi^*M\] where $A$ is an ample $\mbQ$-divisor on $Y$ of relative degree over $Z$ equal to that of $g^*N$ and $0 < \epsilon \ll 1$. By taking $\epsilon \to 0$, we get (\ref{eq:nef_reduction_cbf}) with $\Delta_T$ pseudoeffective.
\end{remark}

\subsection{Non-vanishing and abundance conjectures}
In this subsection, we show the log non-vanishing conjecture for klt pairs in characteristic $p>5$. Furthermore, we prove the log abundance conjecture for klt pairs of non-maximal nef dimension.

We divide the proof into three cases depending on the nef dimension.

\subsubsection{Nef dimension zero}

\begin{proposition} \label{prop:non_vanishing0} Let $(X,\Delta)$ be a projective three-dimensional Kawamata log terminal pair defined over an algebraically closed field $k$ of characteristic $p>5$. Assume that $K_X+\Delta \equiv 0$. Then $K_X+\Delta \sim_{\mbQ} 0$.
\end{proposition}
\begin{proof}
Set $L \defeq K_X+\Delta$. By replacing $X$ by its $\mbQ$-factorialisation, we may assume that $X$ is $\mbQ$-factorial.  The theorem is clear when $\dim \mathrm{alb}(X) = 0$ and it follows from the abundance for surfaces and \cite[Theorem 4.3 or Theorem 1.2]{zhang17} when $\dim \mathrm{alb}(X) = 1$. If $\dim \mathrm{alb}(X) = 2$, then we apply Theorem \ref{theorem:canonical_bundle_formula_main} to the connected-fibres part $g \colon X \to S$ of the Stein factorisation of the albanese morphism. Note that the generic fibre is smooth by Proposition \ref{proposition:general_fibre_smooth}. Therefore, we get a purely inseparable morphism $f \colon T \to S$ and an effective $\mbQ$-divisor $\Delta_T$ on $T$ satisfying
\[
f^*L_S \sim_{\mbQ} tf^*K_S + (1-t)K_T + \Delta_T,
\]
where $L_S$ is such that $g^*L_S \sim_{\mbQ} L$. Since $T$ and $S$ admit generically finite morphisms to an abelian variety, they cannot be uniruled, and hence $\kappa(S) \geq 0$ and $\kappa(T) \geq 0$. Thus $f^*L_S$ is $\mbQ$-effective, and so are $L_S$ and $L$. Last, when $\dim \mathrm{alb}(X) = 3$, the theorem follows from \cite[Theorem C]{daswaldron} (equivalently, run a $K_X$-MMP and apply \cite[Theorem A.1]{daswaldron} or \cite[Theorem 1.1]{zhang17}).
\end{proof} 

\subsubsection{Nef dimension one}

\begin{proposition} \label{prop:non_vanishing1} Let $L$ be a nef $\mbQ$-Cartier $\mbQ$-divisor on a projective three-dimensional variety defined over an algebraically closed field $k$, satisfying $n(L)=1$. Assume that $L \sim_{\mbQ} K_X+\Delta$ for some log canonical pair $(X,\Delta)$, or $k=\Fp$. Then $L$ is semiample.
\end{proposition}
\begin{proof}
Let $\phi \colon X \to Z$ be a nef reduction map of $L$. It is well defined as $\dim Z = 1$ and every rational map to a curve extends to a proper morphism. Furher, $\phi$ is equidimensional. 

We claim that there exists a $\QQ$-divisor $L_Z$ on $Z$ such that $\phi^*L_Z \sim_{\mbQ} L$. Indeed, when $L \sim_{\mbQ} K_X+\Delta$, the abundance for surfaces (see \cite[Theorem 1.1]{tanakaImperfect}) implies $L|_{X_{\mu}} \sim_{\mbQ} 0$ where $X_{\mu}$ is the generic fibre of $X$, and hence the claim follows from Lemma \ref{lemma:descending_nef_line_bundles}. Over $\Fp$ the claim is a consequence of Lemma \ref{lemma:descending_nef_line_bundles_to_curves} and Lemma \ref{lemma:fp}.

We can conlude, as $n(L_Z)=1$ and so the $\mbQ$-divisor $L_Z$ is ample.
\end{proof}

\subsubsection{Nef dimension two} \label{ss:neftwo}

We are left to show the case of nef dimension two. The idea of the proof is the following. Assume for simplicity that the nef reduction map $\psi \colon X \dashrightarrow Z$ of $K_X+\Delta$ is proper. Then, using the canonical bundle formula derived in this article, we can descend $K_X+\Delta$ to an adjoint divisor $K_Z + \Delta_Z$, with $\Delta_Z$ pseudoeffective, up to replacing $Z$ by a purely inseparable cover. By a careful study of adjoint divisors on surfaces (Lemma \ref{lem:bigsurface}), we get that $K_Z+\Delta_Z \equiv D_Z$ for an effective $\mbQ$-divisor $D_Z$, and so $K_X+\Delta \equiv D$ for $D \defeq \psi^*D_Z$. 

Now, we modify $(X,\Delta)$ so that $(X,\Delta)$ is dlt and $\Supp D \subseteq \lfloor \Delta \rfloor$; this is done by replacing $\Delta$ with $\max(\Delta, \Supp D)$ and taking a minimal model of an appropriate dlt modification (Lemma \ref{lem:abundance_dlt_modification}). Then $(K_X+\Delta)|_{\Supp D}$ is semiample by abundance for slc surfaces (see \cite[Theorem 1.3]{waldronlc}), from which we infer that  $(K_Z+\Delta_Z)|_{\Supp D_Z}$ is semiample as well (Lemma \ref{lem:descend}). Consequently, $\kappa(K_Z+\Delta_Z) \geq 0$ by a careful study of adjoint divisors on surfaces (Lemma \ref{lem:bigsurface}). Hence, up to another modification of $(X,\Delta)$ we can assume that $D_Z \sim_{\mbQ} K_Z+\Delta_Z$. In this case, we apply the evaporation technique to show that $K_Z+\Delta_Z$ is big (Lemma \ref{lem:evaporation}), proccuring that $\kappa(K_X+\Delta)\geq 2$. This concludes the proof of the semiampleness of $K_X+\Delta$ thanks to \cite[Theorem 1.3]{waldronabundance}. 

The statements and the proofs of the required results for adjoint divisors on surfaces may be found in Section \ref{s:surfces} (see Lemma \ref{lem:bigsurface}).\\ 

The following lemma is a variation on the standard result used in the proof of the abundance conjecture for threefolds in characteristic zero.

\begin{lemma}[{cf.\ \cite[Lemma 13.2]{kollaretal}}] \label{lem:abundance_dlt_modification} Let $(X,\Delta)$ be a three-dimensional dlt pair defined over an algebraically closed field of characteristic $p>5$ such that $X$ is $\mbQ$-factorial. Suppose that $K_X+\Delta$ is nef and there exists an effective $\mbQ$-divisor $D$ satisfying $D \equiv K_X+\Delta$. Then, there exists a three-dimensional dlt pair $(Y,\Delta_Y)$ such that 
\begin{enumerate}
	\item $K_Y+\Delta_Y$ is nef,
	\item $n(K_Y+\Delta_Y) = n(K_X+\Delta)$, 
	\item $\kappa(K_X+\Delta) \leq \kappa(K_Y+\Delta_Y) \leq \kappa(K_X+\Delta + rD)$ for some $r>0$,
	\item $K_Y+\Delta_Y \equiv E$ for an effective $\mbQ$-divisor $E$ with $\Supp E \subseteq \lfloor \Delta_Y \rfloor$,
	\item $(Y \, \backslash \,\! \Supp E, \Delta_Y) \simeq (X \, \backslash \,\! \Supp D, \Delta)$.
\end{enumerate}
Moreover, if $D \sim_{\mbQ} K_X+\Delta$, then $K_Y+\Delta_Y \sim_{\mbQ} E$ in (4).
\end{lemma}
The pair $(Y,\Delta_Y)$ is nothing but a log minimal model of a dlt blow-up of $(X,\widetilde{\Delta})$, where $\widetilde{\Delta} \defeq \max(\Delta, \Supp D)$. For the convenience of the reader, we append the detailed construction below.
\begin{proof}

We will construct the following commutative diagram
\begin{center}
\begin{tikzcd}
 & W \arrow[dashed]{dl}[swap]{h} \arrow{dr}{g} & \\
Y  & &  X.
\end{tikzcd}
\end{center}
Let $g \colon W \to X$ be a resolution of singularities of $(X, \Supp(\Delta + D))$ with exceptional locus $\mathrm{Exc}$. Set 
\[
\Delta_W \defeq \max(g^*\Delta - K_{W/X}, \Supp(g^{-1}_*D) + \mathrm{Exc}),
\]
that is $\Delta_W$ is constructed from the log pullback of $\Delta$ by replacing all the coefficients of irreducible divisors in $\Supp(g^{-1}_*D) + \mathrm{Exc}$ by one. Since $(X,\Delta)$ is dlt, we can write $K_W+\Delta_W \equiv E_W$, where $E_W$ is an effective $\mbQ$-divisor with $\Supp E_W \subseteq \Supp(g^{-1}_*D) + \mathrm{Exc} \subseteq \lfloor \Delta_W \rfloor$. By replacing $W$ by the output of a $(K_W+\Delta_W)$-MMP over $X$ and using again that $(X,\Delta)$ is dlt, we can assume that $W \, \backslash \,\! \Supp E_W \simeq X \, \backslash \,\! \Supp D$ and $\Supp E_W = \Supp g^*D$.

Now, run a full $(K_W+\Delta_W)$-MMP. Since $K_W+\Delta_W$ is pseudoeffective, it terminates with a minimal model $h \colon W \dashrightarrow Y$; in particular $K_Y+\Delta_Y$ is nef for $\Delta_Y \defeq h_*\Delta$. Furthermore, $K_Y+\Delta_Y \equiv E$ and $\Supp E \subseteq \lfloor \Delta_Y \rfloor$, where $E \defeq h_*E_W$. Hence, (1) and (4) hold. Since $h$ and $g$ are isomorphisms outside of $E_W$, we get (5).

To show (2) and (3), we find $h' \colon W' \to W$ such that $h \circ h'$ is a well defined morphism. Replace $h$ by $h \circ h'$ and $g$ by $g \circ h'$. Then
\[
g^*(K_X+\Delta) \leq h^*(K_Y+\Delta_Y) \leq g^*(K_X+\Delta + rD)
\]
for some $r>0$ . The first inequality follows from the negativity lemma: $g^*(K_X+\Delta) \leq h^*h_*g^*(K_X+\Delta) \leq h^*(K_Y+\Delta_Y)$, while the second inequality is a consequence of (5) as $\Supp g^*D = \Supp h^*E$. Therefore, 
\begin{align*}
&n(K_Y+\Delta_Y) = n(K_X+\Delta),  \text{ and } \\
&\kappa(K_X+\Delta) \leq \kappa(K_Y+\Delta_Y) \leq \kappa(K_X+\Delta + rD),
\end{align*}
concluding the proof of the lemma.
\end{proof}

\begin{lemma} \label{lem:descend} Let $X$ be a three-dimensional projective variety defined over an algebraically closed field $k$ of characteristic $p>0$, let $\psi \colon X \to Y$ be an equidimensional contraction onto a surface $Y$, let $M$ be a Cartier divisor on $Y$, and let $C$ be any reduced divisor on $Y$. Assume that $\psi^*M|_{S}$ is semiample for the reduction $S$ of $\psi^{-1}(C)$. Then $M|_{C}$ is semiample.
\end{lemma}
Note that $C$ is not necessarily irreducible.
\begin{proof}
Let 
\begin{center}
\begin{tikzcd}
S \arrow[bend left = 30]{rr}{\psi|_S} \arrow{r}{v} & \widetilde{C} \arrow{r}{u} & C
\end{tikzcd}
\end{center}
be the Stein factorisation of $\psi|_S$. Since $(\psi|_S)^*(M|_C)=\psi^*M|_S$ is semiample and $v_*\mcO_S = \mcO_{\widetilde{C}}$, the projection formula shows that $u^*(M|_C)$ is semiample. Moreover, by the Zariski connectedness theorem, $\psi$ has connected fibres, and so $\psi|_S$ and $u$ have connected fibres as well. This in turn implies that $u$ is a universal homeomorphism; indeed it is a bijective morphism between curves defined over an algebraically closed field (see \cite[Lemma 5.2]{kollar97}). Thus $M|_C$ is semiample by Lemma \ref{lem:Keel_universal_mor}.
\end{proof}

We are ready to tackle the case of nef dimension two. For the clarity of the proof we first show an MMP-free version of Proposition \ref{prop:non_vanishing2}.
\begin{lemma} \label{lem:non_vanishing2} Let $(X,\Delta)$ be a projective three-dimensional log pair defined over an algebraically closed field $k$ of characteristic $p>3$ such that the coefficients of $\Delta$ are at most one. Assume that $L \defeq K_X+\Delta$ is nef and $n(L)=2$. Then the following hold:
\begin{enumerate}
	\item there exists an effective $\mbQ$-divisor $D$ such that $L \equiv D$,
	\item if $L|_{\Supp D} \sim_{\mbQ} 0$ for some $D$ as above, then $\kappa(L) \geq 0$, and
	\item if $L|_{\Supp D} \not \equiv 0$, or $L|_{\Supp D} \sim_{\mbQ} 0$ and $L \sim_{\mbQ} D$, then $\kappa(L)= 2$.
\end{enumerate}
\end{lemma}
\begin{proof}
Let $\psi \colon X \dashrightarrow Z$ be a nef reduction map for $L$. By the assumptions on the characteristic and by Proposition \ref{proposition:general_fibre_smooth}, the generic fibre $X_{\mu}$ is smooth. 

Let
\begin{center}
\begin{tikzcd}
& X \arrow[dashed]{d}{\phi} & W \arrow[swap]{l}{g} \arrow{ld}{\psi} \\
T \arrow{r}{f} & Z & 
\end{tikzcd}
\end{center}  
be a diagram as in Proposition \ref{prop:nef_reduction_cbf} with $L_Z$ being a $\mbQ$-divisor on $Z$ such that $g^*L \sim_{\mbQ} \psi^*L_Z$ and $f^*L_Z \sim_{\mbQ} K_T + \Delta_T - M$ for a pseudo-effective $\mbQ$-divisor $\Delta_T$ and an effective $\mbQ$-divisor $M$ satisfying $f^*L_Z|_{\Supp M} \equiv 0$. In particular, Lemma \ref{lem:bigsurface} implies that there exists an effective $\mbQ$-divisor $D_T$ on $T$ such that $f^*L_Z \equiv D_T$. Take 
\begin{align*}
 D_Z &\defeq {\textstyle\frac{1}{\deg f}}\, f_*D_T, \\
 D_W &\defeq \psi^*D_Z, \text{ and } \\ 
 D &\defeq g(D_Y).
\end{align*}
Therewith, $L_Z \equiv D_Z$ and $L \equiv D$, thus (1) holds (here we use Remark \ref{rem:descend_num_equiv}).

Now, assume that $L \equiv D'$ for a possibly different effective $\mbQ$-divisor $D'$ on $X$, and let $D'_W \defeq g^*D'$. Since $D' \equiv L$ and $D'$ is effective, we get that $D'|_{X_{\mu}} \sim_{\mbQ} 0$, and so Lemma \ref{lemma:descending_nef_line_bundles} implies the existence of an effective $\mbQ$-divisor $D'_Z$ on $Z$ such that $D'_W \sim_{\mbQ} \psi^*D'_Z$. Since $\psi$ is a contraction, we can pick $D'_Z$ so that in fact $D'_W = \psi^*D'_Z$. Set $D'_T \defeq f^*D'_Z$.

 If $L|_{\Supp D'} \not \equiv 0$, then $L_Z \cdot D'_Z > 0$ and $L_Z$ is big. In particular, $\kappa(L) \geq 2$. Thus, we can assume that $L|_{\Supp D'} \equiv 0$. To show (2) and (3), suppose $L|_{\Supp D'} \sim_{\mbQ} 0$. By Lemma \ref{lem:descend}, we get $L_Z|_{\Supp D'_Z} \sim_{\mbQ} 0$, and so $f^*L_Z|_{\Supp D'_T} \sim_{\mbQ} 0$. Hence, we can conclude the proof by Lemma \ref{lem:bigsurface}.
\end{proof}

\begin{remark} \label{rem:descend_num_equiv} Let $f \colon X \to Y$ be a proper and generically purely inseparable morphism between $\mbQ$-factorial varieties and let $L$ be a divisor on $X$ such that $L \equiv_f 0$. Then $L = \frac{1}{\deg f} f^*f_*L$ by the negativity lemma. In particular, if $D_1 \equiv D_2$ for some divisors $D_1$, $D_2$ on $X$, then $f_*D_1 \equiv f_*D_2$.
\end{remark}

\begin{proposition} \label{prop:non_vanishing2} Let $(X,\Delta)$ be a projective three-dimensional Kawamata log terminal pair defined over an algebraically closed field $k$ of characteristic $p>5$. Assume that $K_X+\Delta$ is nef and $n(K_X+\Delta)=2$. Then $K_X+\Delta$ is semiample.
\end{proposition}

\begin{proof}
We only need to show that $\kappa(K_X+\Delta) = 2$; indeed in this case the semiampleness of $K_X+\Delta$ follows from \cite[Theorem 1.3]{waldronabundance}. Note that $\kappa(K_X+\Delta) \leq n(K_X+\Delta) = 2$. By replacing $X$ by its $\mbQ$-factorialisation, we can assume that $X$ is $\mbQ$-factorial.

By Lemma \ref{lem:non_vanishing2}, there exists an effective $\mbQ$-divisor $D$ satisfying $K_X+\Delta \equiv D$. 
Let $(Y,\Delta_Y)$ and $E_Y$ be as in Lemma \ref{lem:abundance_dlt_modification}, that is for some $r>0$:
\begin{itemize}
	\item $K_Y+\Delta_Y$ is nef,
	\item $n(K_Y+\Delta_Y)=n(K_X+\Delta)$ and $\kappa(K_Y+\Delta_Y) \leq \kappa(K_X + \Delta + rD)$,
	\item $K_Y+\Delta_Y \equiv E_Y$ where $E_Y$ is an effective $\mbQ$-divisor such that $\Supp E_Y \subseteq \lfloor \Delta_Y \rfloor$.
\end{itemize}

By \cite[Theorem 1.3]{waldronlc}, $(K_Y+\Delta_Y)|_{\lfloor \Delta_Y \rfloor}$ is semiample, and therefore $(K_Y+\Delta_Y)|_{\Supp E_Y}$ is semiample as well. By Lemma \ref{lem:non_vanishing2} applied to $(Y,\Delta_Y)$ and $E_Y$, we see that $\kappa(K_Y+\Delta_Y) \geq 0$.

We claim that in fact $\kappa(K_Y+\Delta_Y)\geq 2$. 
To this end, we apply Lemma \ref{lem:abundance_dlt_modification} to $(Y,\Delta_Y)$ and a $\mbQ$-section of $K_Y+\Delta_Y$,  obtaining a dlt pair $(Z,\Delta_{Z})$ satisfying
\begin{itemize}
	\item $K_{Z}+\Delta_{Z}$ is nef,
	\item $n(K_{Z}+\Delta_{Z})=n(K_Y+\Delta_Y)$ and $\kappa(K_{Z}+\Delta_{Z}) = \kappa(K_Y+\Delta_Y)$,
	\item $K_{Z}+\Delta_{Z} \sim_{\mbQ} E_{Z}$ where $E_{Z}$ is an effective $\mbQ$-divisor such that $\Supp E_{Z} \subseteq \lfloor \Delta_{Z} \rfloor$.
\end{itemize}
Arguing as above, we get that $(K_{Z}+\Delta_{Z})|_{\Supp E_{Z}}$ is semiample. Lemma \ref{lem:non_vanishing2} (3) applied to $(Z,\Delta_{Z})$ and $E_{Z}$ implies $\kappa(K_Y+\Delta_Y) = \kappa(K_{Z}+\Delta_{Z}) = 2$.

We are left to show $\kappa(K_X+\Delta)=2$. Since $\kappa(K_Y+\Delta_Y)=2$, we have $\kappa(K_X+\Delta + rD) \geq 2$. However, given that $K_X+\Delta \equiv D$, this is only possible when $(K_X+\Delta)|_D \not \equiv 0$. By (3) in Lemma \ref{lem:non_vanishing2}, we can conclude the proof of the proposition.   
\end{proof}

\subsubsection{Proof of Theorem \ref{thmi:non_vanishing} and Theorem \ref{thmi:abundance}}

\begin{theorem}[{Theorem \ref{thmi:non_vanishing}}] \label{thm:nonvanishing_main}  Let $(X,\Delta)$ be a projective Kawamata log terminal pair of dimension three defined over an algebraically closed field $k$ of characteristic $p>5$. Assume that $K_X+\Delta$ is pseudo-effective. Then $\kappa(K_X+\Delta) \geq 0$.
\end{theorem}
The proof follows from Proposition \ref{prop:non_vanishing0}, Proposition \ref{prop:non_vanishing1}, Proposition \ref{prop:non_vanishing2}, and \cite[Theorem 1.1]{xuzhang18} by ways of \cite[Lemma 2.4]{kkm94}. For the convenience of the reader we append the argument of \cite{kkm94} below. 
\begin{proof}
Since the canonical ring is preserved under any MMP, it is enough to show the theorem for a minimal model of $(X,\Delta)$. In particular, we can replace $(X,\Delta)$ by a minimal model (see \cite[Theorem 1.2]{birkar13}) and assume that $K_X+\Delta$ is nef. Moreover, by replacing $(X,\Delta)$ by its sub-terminalisation (see \cite[Theorem 1.7]{birkar13}), we can assume that $X$ is terminal and $\mbQ$-factorial. 

Let $0 \leq \lambda \leq 1$ be the smallest number such that $K_X+ \lambda \Delta$ is nef. Since $K_X+\Delta - (K_X+\lambda \Delta)$ is $\mbQ$-effective, it is enough to show that $K_X+\lambda \Delta$ is $\mbQ$-effective. If $\lambda = 0$, then the result follows from \cite[Theorem 1.1]{xuzhang18}, thus we can assume $\lambda > 0$.

We have $\frac{1}{\lambda}(K_X+\lambda \Delta) = (K_X+\Delta) + \frac{1-\lambda}{\lambda}K_X$, and so there exists a $K_X$-extremal ray $R$ such that $(K_X+\lambda \Delta)\cdot R=0$ by \cite[Lemma 2.1]{kkm94} (though stated only for $\mathbb{C}$, the result is a direct consequence of the existence of bounds on lengths of extremal rays, true for klt pairs in positive characteristic by \cite[Theorem 1.1]{bw14}). In particular, $\lambda \in \mbQ$.

If $n(K_X+\lambda \Delta) \leq 2$, then $K_X+\lambda \Delta$ is semiample by Proposition \ref{prop:non_vanishing0}, Proposition \ref{prop:non_vanishing1}, and Proposition \ref{prop:non_vanishing2}, thus we can assume $n(K_X+\lambda \Delta)=3$. 

Let $f \colon X \dashrightarrow Y$ be the contraction of the ray $R$ (or the corresponding flip if the contraction is small). Since $n(K_X+\lambda \Delta)=3$, the map $f$ is birational. Set $\Delta_Y = f_* \Delta$. It is enough to show that the nef $\mbQ$-divisor $K_Y + \lambda \Delta_Y$ is $\mbQ$-effective, hence we can replace $(X,\Delta)$ by $(Y,\lambda \Delta_Y)$ and proceed as above by choosing a new $\lambda$. We can conclude the proof that $\kappa(K_X+\Delta)\geq 0$ as any sequence of flips and divisorial contractions for terminal varieties terminates.
\end{proof}

\begin{remark} As it has been communicated to us by Tanaka, Theorem \ref{thm:nonvanishing_main} holds in the relative setting for projective contractions $\phi \colon X \to Z$ such that $K_X+\Delta$ is relatively pseudo-effective. Indeed, by arguing as above we can assume that $K_X+\Delta$ is relatively nef. By \cite{hnt}, we can compactify $\phi$ to a morphism of projective varieties, and, by running the MMP over the base, we can assume that $K_X+\Delta$ is still relatively nef. Using the cone theorem, we can find an ample divisor $A$ on $Z$ such that $K_X+\Delta + \phi^*A$ is globally nef. By \cite[Theorem 1]{tanaka_relative17}, we can assue that $(X,\Delta+\phi^*A)$ is klt. Hence, $\kappa(K_X+\Delta+\phi^*A)\geq 0$ by Theorem \ref{thm:nonvanishing_main}, and so $K_X+\Delta$ is relatively $\mbQ$-effective.
\end{remark}

\begin{proof}[Proof of Theorem \ref{thmi:abundance}]
This follows directly from Proposition \ref{prop:non_vanishing0}, Proposition \ref{prop:non_vanishing1}, and Proposition \ref{prop:non_vanishing2}.
\end{proof}

Further, we obtain the following result.
\begin{corollary} \label{cori:albabundance} Let $(X,\Delta)$ be a projective Kawamata log terminal pair of dimension three defined over an algebraically closed field $k$ of characteristic $p>5$. Assume that $\dim \mathrm{Alb}(X) > 0$ and $K_X+\Delta$ is nef. Then $K_X+\Delta$ is semiample. 
\end{corollary}
\begin{proof}
Let $g \colon X \to S$ be the connected-fibres part of the Stein factorisation of the albanese morphism, and let $G$ be the generic fibre of $g$. Then $(K_X+\Delta)|_G$ is semiample by abundance for curves and surfaces (see \cite[Theorem 1.1]{tanakaImperfect}), and so $\kappa((K_X+\Delta)|_G) = n((K_X+\Delta)|_G)$.

In \cite[Theorem 1.2]{zhang17}, it is shown that $K_X+\Delta$ is semiample, except when:
\begin{itemize}
	\item $\dim S = 2$ and $n((K_X+\Delta)|_G)=0$, or
	\item $\dim S =1$ and $n((K_X+\Delta)|_G)=1$.
\end{itemize}
In both of these cases $n(K_X+\Delta) < 3$, so we can conclude the proof by Theorem \ref{thmi:abundance}.
\end{proof}

\begin{remark} \label{remark:logabundance}
The proof of the log abundance conjecture for three-dimensional klt pairs in characteristic zero (see \cite{kkm94}) consists of three steps: showing abundance for terminal varieties, showing that any effective nef divisor on a klt Calabi-Yau variety is semiample, and showing log abundance for when $n(K_X+\Delta)\leq 2$. The proof of the second step is quite similar in nature to the proof of the first one. Note that Theorem \ref{thmi:abundance} concludes the third step of the proof of log abundance for threefolds in positive characteristic. 
\end{remark}

\subsection{Low characteristic MMP}
In this subsection, we show Theorem \ref{thmi:bpf} and Theorem \ref{thmi:nonvanishing_bpf}.

\begin{lemma} \label{lem:non_vanishing_low} Let $(X,\Delta)$ be a projective three-dimensional log pair defined over an algebraically closed field $k$ of characteristic $p>2$ such that $\Delta$ has coefficients at most one. Let $L$ be a $\mbQ$-Cartier $\mbQ$-divisor on $X$ such that $L - (K_X+\Delta)$ is nef and big, and $n(L)=2$. Then the following hold:
\begin{itemize}
\item if $\lfloor \Delta \rfloor = 0$, then $\kappa(L)=2$,
\item if $k=\Fp$, then $L$ is semiample.
\end{itemize}
\end{lemma}
\begin{proof}
The proof follows from Proposition \ref{prop:nef_reduction_cbf}, Remark \ref{remark:cbf_pseudoeffective}, and Lemma \ref{lem:bigsurface} analogously to the proof of Lemma \ref{lem:non_vanishing2}.
\end{proof}

Note that in general the log canonical base point free theorem is false over uncountable algebraically closed fields (see \cite{MNW15,NW17}). 
\begin{theorem} (Theorem \ref{thmi:bpf}) \label{thm:bpf} Let $(X,\Delta)$ be a projective three-dimensional log pair defined over $\Fp$ for $p>2$ and let $L$ be a nef $\mbQ$-Cartier $\mbQ$-divisor such that $L-(K_X+\Delta)$ is nef and big. Assume that
\begin{itemize}
	\item $(X,\Delta)$ is log canonical, or
	\item the coefficients of $\Delta$ are at most one and each irreducible component of $\lfloor \Delta \rfloor$ is normal.
\end{itemize}
Then $L$ is semiample.
\end{theorem}
\begin{proof}
If $L$ is big, then this follows from \cite[Theorem 1.1]{MNW15}. Thus,  we can assume that $L$ is not big, and hence Theorem \ref{theorem:birkar} implies that $n(L)\leq 2$. When $n(L)=2$, we apply Lemma \ref{lem:non_vanishing_low}; when $n(L)=1$, we apply Proposition \ref{prop:non_vanishing1}; and when $n(L)=0$, we apply Lemma \ref{lemma:fp}.
\end{proof} 

\begin{theorem}[{Theorem \ref{thmi:nonvanishing_bpf}}] \label{thm:nonvanishing_bpf} Let $(X,\Delta)$ be a three-dimensional projective klt pair defined over an algebraically closed field of characteristic $p>2$ and let $L$ be a nef $\mathbb{Q}$-Cartier $\mathbb{Q}$-divisor such that $L - (K_X+\Delta)$ is nef and big. When $n(L)=0$, assume $\dim \mathrm{alb}(X) \neq 1$. Then $\kappa(L) = n(L)$.
\end{theorem}
\begin{proof}
When $n(L)=3$, this is Theorem \ref{theorem:birkar}; when $n(L)=2$, this is Lemma \ref{lem:non_vanishing_low}; and  when $n(L)=1$, this is Proposition \ref{prop:non_vanishing1}. In the case of $n(L)=0$, the proof is analogous to that of Proposition \ref{prop:non_vanishing0} after one notices that when $n(L)=0$, we have $\dim \mathrm{alb}(X) \leq 2$, since by bend-and-break $X$ is covered in this case by rational curves (see \cite[II, Theorem 5.14]{kollar96}). Taking a $\mbQ$-factorialisation is not necessary, because we assume that $\dim \mathrm{alb}(X) \neq 1$ when $n(L)=0$.
\end{proof}

\section{Positivity of adjoint divisors on surfaces} \label{s:surfces}
The following lemma was a fundamental component in the above applications of the canonical bundle formula. On the first reading, the reader might start by assuming that $M=0$; the reason we consider a non-zero $M$ is due to the fact that nef reduction maps need not be proper. 
\begin{lemma} \label{lem:bigsurface} Let $X$ be a normal $\mbQ$-factorial projective surface defined over an algebraically closed field $k$ of characteristic $p>2$, let $\Delta$ be a pseudo-effective $\mbQ$-divisor, let $L$ be a nef $\mbQ$-divisor such that $n(L)=2$, and let $M$ be an effective $\mbQ$-divisor such that $L|_{\Supp M} \equiv 0$. Suppose that 
\[
L \sim_{\mbQ} K_X + \Delta - M.
\] 
Then the following hold:
\begin{enumerate}
	\item there exists an effective $\mbQ$-divisor $D$ such that $L \equiv D$,
	\item if $\Delta$ is big, then $L$ is big,
	\item if $k= \Fp$, then $L$ is semiample.
	\item if $L|_{\Supp D}$ is semiample for an effective $\mbQ$-divisor $D$ such that $D \equiv L$, then $\kappa(L)\geq 0$. Similarly, if $L|_{\Supp D}$ is semiample for an effective $D \sim_{\mbQ} L$, then $L$ is big.
\end{enumerate} 
\end{lemma}
In the proof we implicitly and freely use Remark \ref{rem:descend_num_equiv} for birational morphisms.
\begin{proof}
By replacing $X$ (and $\Delta$, $M$, $D$) by the minimal resolution of singularities (by the log-pullback or pullback, accordingly), we can assume that $X$ is smooth. By adding a certain multiple of $L$ to $\Delta$, we can suppose henceforth that $L$ is Cartier. Further, by means of the MMP we can contract any $K_X$-negative extremal ray $R$ satisfying $L \cdot R = 0$. With $X$ replaced by the output of such a partial MMP, we can assume that $K_X + 3L$ is nef thanks to the cone theorem (\cite[Proposition 3.15]{tanaka12}, cf.\ the proof of Lemma \ref{lemma:pseudoeffective_surfaces}). We use that $n(L)=2$  to ensure that the output of the MMP is of dimension two. Note, that in Case (4) the assumption that $L|_{\Supp D}$ is semiample might no longer hold, but we can assume that there exists a birational morphism of smooth projective surfaces $f \colon Y \to X$ such that $f^*L|_{\Supp f^*D}$ is semiample. Last, we can assume that $\Delta'$ and $M$ have no irreducible components in common, where $\Delta = \Delta' + N$ is the Zariski decomposition with $\Delta'$ effective and $N$ nef. Note that $N$ is big and nef when $\Delta$ is big.

We start by showing that $M$ is nef. To this end, we pick a curve $C \subseteq \Supp M$. By assumptions, $L \cdot C = 0$ and so
\[
0 = (K_X + 3L + \Delta - M) \cdot C \geq -M \cdot C.
\]
Thus $M \cdot C \geq 0$ and $M$ is nef. Moreover, $M^2=0$, as otherwise $M$ is big thereby contradicting $n(L)=2$, and so $M|_{\Supp M} \equiv 0$.

First, we deal with (3), that is the case of $k=\Fp$. By Theorem \ref{thm:mnw_surface}, the $\mbQ$-divisor $L + M \sim_{\mbQ} K_X + \Delta$ is semiample. Since $n(L)=2$, the associated semiample fibration is birational, and as $(L+M)|_M \equiv 0$ it must contract $M$. This is only possible when $M=0$, hence $L$ is semiample.\\

We turn our attention to (1) and (2). First, we will show that $L$ is big if $\Delta$ is big or if $X$ is not isomorphic to a rational surface or a ruled surface over an elliptic curve. By contradiction, we can assume that $L$ is not big, i.e.\ $L^2=0$. Since $n(L)=2$, we must have $\kappa(L) \leq 0$. Furthermore, 
\[
0 = L^2 = (K_X + \Delta - M) \cdot L = (K_X+\Delta) \cdot L 
\]
and so $K_X \cdot L \leq 0$. If $\Delta$ is big, then $K_X \cdot L <0$. 

By Serre duality $H^2(X, \mcO_X(mL)) = 0$, and so Riemann-Roch yields
\[
H^0(X, \mcO_X(mL)) = H^1(X, \mcO_X(mL)) - \frac{1}{2}mL \cdot K_X + \chi(\mcO_X).
\]
Since $\kappa(L)\leq 0$, we have $L \cdot K_X = 0$. This leads to a contradiction when $\Delta$ is big, and so (2) holds. 

Let us note that since $K_X+3L$ is nef, we have $(K_X+3L)^2 \geq 0$, and so $K_X^2\geq 0$ (recall that $L^2=K_X \cdot L =0$). We take $\pi \colon X \to X^{\mathrm{min}}$ to be a morphism to a minimal model $X^{\mathrm{min}}$ of $X$ and proceed by classification of surfaces. We cannot have $\kappa(X)\geq 1$ as $K_X \cdot L = 0$ and $n(L)=2$. If $\kappa(X)=0$, then $L \cdot \Exc(\pi) = 0$ since $\Supp K_X = \Exc(\pi)$. In particular, $X$ is minimal given that $K_X+3L$ is nef. We shall now proceed by a case-by-case analysis (cf.\ \cite[Proof of Theorem 1.4 (Case 1.3)]{MNW15}).

If $X$ is a K3 surface, then $\chi(\mcO_X)\geq 2$ and $h^0(X,\mcO_X(mL)) \geq 2$ by the Riemann-Roch theorem as above. In particular $h^0(X,\mcO_X(kmL)) \geq k+1$ for $k \in \mbN$ which is impossible as $\kappa(L)\leq 0$. If $X$ is an Enriques surface, then it admits an \'etale double cover $h \colon S \to X$ by a K3 surface $S$ and by Riemann-Roch $h^0(S, \mcO_X(mh^*L)) \geq 2$. This is again impossible.

If $X$ is an abelian variety, then $L$ is numerically equivalent to some semiample divisor $D$ (see \cite[Proposition 3.10]{MNW15}). Since $n(D)=2$, we have that $D$ is big, and so $L$ is big as well. If $X$ is hyperelliptic, then it admits a finite cover $h \colon S \to X$ by an abelian variety $S$, and the same argument gives that $h^*L$ is big.

Last if $X$ is a quasi-hyperelliptic surface, then it admits a surjective morphism $h \colon S \defeq C \times \mbP^1 \to X$, where $C$ is an elliptic curve. Since $\Pic(S) \simeq \Pic(C) \times \mbZ$ and $n(h^*L)=2$, we get that $h^*L$ is ample. 

Therefore, we can assume that $\kappa(X)=-\infty$. Since $X$ is irrational, $X^{\mathrm{min}}$ is a ruled surface over a curve $C$ of genus $g(C)\geq 1$. By theory of ruled surfaces, $K_{X^{\mathrm{min}}}^2 = 8(1-g(C)) \leq 0$. Since $K_X^2 \geq 0$, we get that in fact $X = X^{\mathrm{min}}$ and $g(C)=1$. 

To finish off the proof of (1) we need to show that $L$ is numerically equivalent to an effective divisor if $X$ is rational or is a ruled surface over an elliptic curve. In the latter case, this follows from \cite[Proposition 3.13]{MNW15}. In the former, $L$ is in fact $\mbQ$-effective; indeed $\chi(\mcO_X)>0$, and so by Riemann-Roch $\kappa(L)\geq 0$.\\

We are left to show (4). By the above proof, we can assume that $X$ is rational or is a ruled surface over an elliptic curve. Let $f \colon Y \to X$ be the birational morphism such that $f^*L|_{\Supp f^*D} \sim_{\mbQ} 0$. Set $L_Y \defeq f^*L$ and $D_Y \defeq f^*D$. Replacing $L_Y$ and $D_Y$ by $mL_Y$ and $mD_Y$, respectively, for divisible enough $m\in \mbN$, we can assume that $D_Y$ is an effective divisor. Then Lemma \ref{lem:Keel_universal_mor} implies that $(L_Y)|_{D_Y} \sim_{\mbQ} 0$. Further, since $L^2 = L\cdot K_X =0$, we have $L_Y^2 = L_Y \cdot K_Y = 0$. 

First, we consider the case of $D_Y \equiv L_Y$. Assume by contradiction that $\kappa(L_Y) = -\infty$. Since $X$ is rational or is a ruled surface over an elliptic curve, we have $\chi(\mcO_Y) \in \{0,1\}$ and the Riemann Roch theorem
\[
H^0(Y, kL_Y-D_Y) = H^1(Y, kL_Y-D_Y) + \chi(\mcO_Y)
\]
implies $H^1(Y, kL_Y-D_Y)=0$ for all $k \in \mbZ_{>1}$. Thus, for the exact sequence
\[
H^0(Y, kL_Y) \xrightarrow{\mathrm{res}} H^0(D_Y, (kL_Y)|_{D_Y}) \to H^1(Y, kL_Y-D_Y),
\] 
we have that $\mathrm{res}$ is surjective. By taking $k \in \mbN$ such that $kL_Y|_{D_Y} \sim 0$, we get $H^0(Y,kL_Y) \neq 0$, concluding the proof in this case.

When $D_Y \sim_{\mbQ} L_Y$, we apply Lemma \ref{lem:evaporation} and obtain that $\kappa(L_Y) \geq 1$. Since $n(L_Y)=2$, this implies that $L_Y$ is big.
\end{proof}

\begin{remark} With notation as in the above lemma, $L$ need not be $\mbQ$-effective when $\Delta$ is not big and $k \not \simeq \Fp$. Indeed, consider $X \defeq \mathbb{P}_C(\mcO_C \oplus \mcL)$, where $C$ is an elliptic curve and $\mcL \in \Pic^0(C)$ is a non-torsion line bundle. Let $C_1, C_2 \subseteq X$ be two sections corresponding to $\mcO_C$ and $\mcL$. Then $\mcO_X(C_1-C_2)|_{C_1} \sim \mcL$, and so $C_1-C_2 \equiv 0$. Since $K_X \sim -C_1 - C_2$, we have $C_1-C_2 \sim K_X + 2C_1$.
\end{remark} 
Note that a big part of the proof is valid even if $p=2$; in fact the assumption that $p>2$ is only used when $X$ is birational to an Enriques surface.

The following lemma, which was used in the above result, is a staple among applications of the so called ``evaporation technique''. Its proof is exactly the same as that of Keel (see \cite{totaro09}), but since we work with non-integral curves we decided to append the whole argument for the convenience of the reader.
\begin{lemma}[{cf.\ \cite[Theorem 2.1]{totaro09}}] \label{lem:evaporation} Let $X$ be a smooth projective surface and let $D$ be a nef effective divisor such that $D\cdot K_X=0$ and $D|_{\Supp D}$ is semiample. Then $\kappa(D)\geq 1$. 
\end{lemma}
\begin{proof}
By Lemma \ref{lem:Keel_universal_mor}, $D|_D$ is semiample. If $D^2>0$, then $D$ is big, so we can assume that $D^2=0$ and $D|_D \sim_{\mbQ} 0$.

By the Riemann-Roch theorem
\[
\chi(\mcO_X(kD)) = \frac{1}{2}kD \cdot (kD - K_X) + \chi(\mcO_X) = \chi(\mcO_X)
\]
for all $k \in \mbZ$. Take $m \in \mbN$ to be the smallest number such that $mD|_D \sim 0$ and consider the following exact sequence for any $0 \leq k \leq m$:

\vspace{2pt}
\begin{tikzcd}
  H^0(X,\mcO_X((k-1)D)) \rar & H^0(X, \mcO_X(kD)) \rar
             \ar[draw=none]{d}[name=X, anchor=center]{}
    & H^0(D, \mcO_D(kD)) \ar[rounded corners,
            to path={ -- ([xshift=2ex]\tikztostart.east)
                      |- (X.center) \tikztonodes
                      -| ([xshift=-2ex]\tikztotarget.west)
                      -- (\tikztotarget)}]{dll}[at end]{} \\      
  H^1(X, \mcO_X((k-1)D)) \rar & H^1(X, \mcO_X(kD)) \rar & H^1(D, \mcO_D(kD)).
\end{tikzcd}\vspace{2pt} \\

Since $\chi(\mcO_X((k-1)D)) = \chi(\mcO_X(kD))$, we get that $\chi(\mcO_D(kD))=0$. Thus, if $0<k<m$, then 
\[
H^1(D, \mcO_D(kD))=H^0(D, \mcO_D(kD))=0.
\]
In particular, the above exact sequence provides a natural isomorphism
\[
H^1(X, \mcO_X((k-1)D)) \simeq H^1(X, \mcO_X(kD)).
\]
Therefore,
\[
H^1(X,\mcO_X) \simeq H^1(X,\mcO_X(D)) \simeq \ldots \simeq H^1(X,\mcO_X((m-1)D)).
\]

By \cite[Proposition 12 and Section 9]{Serre58}, there exists a finite morphism $f \colon Y \to X$ such that $f^*H^1(X,\mcO_X)=0$, and so $f^*H^1(X,\mcO_X((m-1)D)) = 0$ as well. Now, take a nonzero section $s \in H^0(D, \mcO_D(mD))$ and consider the following commutative diagram with $D_Y \defeq f^*D$:
{ \footnotesize
\begin{center}
\begin{tikzcd}
H^0(Y, \mcO_Y(mD_Y)) \arrow{r}{\mathrm{res}_Y} & H^0(D_Y, \mcO_{D_Y}(mD_Y)) \arrow{r}{\delta_Y} & H^1(Y, \mcO_Y((m-1)D_Y))  \\
H^0(X, \mcO_X(mD)) \arrow{u}{f^*} \arrow{r}{\mathrm{res}} & H^0(D, \mcO_D(mD)) \arrow{u}{f^*} \arrow{r}{\delta} & H^1(X, \mcO_X((m-1)D)). \arrow{u}{f^*} 
\end{tikzcd}
\end{center}
}
Since $f^*H^1(X,\mcO_X((m-1)D))=0$, we have that $\delta_Y(f^*s)=0$, and so $f^*s$ lies in the image of $\mathrm{res}_Y$. Thus $h^0(Y, \mcO_Y(mD_Y)) \geq 2$ and $\kappa(D_Y)\geq 1$. By a standard argument (cf. \cite[Lemma 2.10]{keel99}),  we have $\kappa(D)\geq 1$. 
\end{proof}

\section*{Acknowledgements}
I would like to express my special gratitude to Paolo Cascini for his substantial help, encouragement, and support. 

Further, I would like to thank Hiromu Tanaka for numerous discussions, encouragement, and advice on proving log non-vanishing. I also thank Yoshinori Gongyo, Diletta Martinelli, Yusuke Nakamura, Johannes Nicaise, Zsolt Patakfalvi, Joe Waldron, Chenyang Xu, and Lei Zhang for comments and helpful suggestions. The paper has been motivated by the work \cite{MNW15} started at the Pragmatic research school in Catania 2013.

The author was supported by the Engineering and Physical Sciences Research Council [EP/L015234/1].

\bibliographystyle{amsalpha}
\bibliography{LibraryExtracted2}

\end{document}